\documentclass[a4paper,12pt]{article}
\usepackage{amsmath,
            amssymb,
            latexsym,
            amsthm,
            enumerate,
            enumitem,
            mathrsfs,
            mathtools,
            nicefrac,
            dsfont}

\usepackage[left=25mm,right=25mm,top=15mm,bottom=15mm,includeheadfoot]{geometry}
\usepackage{xcolor}
\usepackage{graphicx}
\usepackage{epstopdf}
\usepackage[nocompress]{cite}

\usepackage{comment}

\usepackage[T1]{fontenc}
\usepackage[utf8]{inputenc}
\usepackage[english]{babel}

\usepackage{hyperref}
\usepackage[capitalise,noabbrev]{cleveref}
\crefname{enumi}{item}{items}

\newcommand*{\ol}{\overline}

\newcommand{\dx}{\mathop{}\!\mathrm{d}}

\newcommand*{\bB}{\mathbf{B}}
\newcommand*{\bC}{\mathbf{C}}

\newcommand*{\bK}{\mathbf{K}}

\newcommand*{\bN}{\mathbf{N}}

\newcommand*{\bX}{\mathbf{X}}

\newcommand{\fD}{\mathfrak{D}}

\newcommand{\fM}{\mathfrak{M}}

\newcommand{\fb}{\mathfrak{b}}

\newcommand{\fh}{\mathfrak{h}}

\newcommand{\fr}{\mathfrak{r}}

\newcommand*{\cB}{\mathcal{B}}

\newcommand*{\cD}{\mathcal{D}}

\newcommand*{\cF}{\mathcal{F}}

\newcommand*{\cL}{\mathcal{L}}

\newcommand*{\cR}{\mathcal{R}}

\newcommand*{\cW}{\mathcal{W}}

\newcommand*{\N}{\mathbb{N}}

\newcommand*{\E}{\mathbb{E}}

\newcommand*{\R}{\mathbb{R}}

\newcommand{\with}{\curvearrowleft}

\newcommand*{\eps}{\varepsilon}

\newcommand{\mdist}{\bar{r}}

\newcommand*{\tr}{\operatorname{tr}}
\newcommand*{\Id}{\operatorname{Id}}

\newcommand*{\diam}{\operatorname{diam}}
\newcommand*{\dist}{\operatorname{dist}}

\newcommand{\ind}{\mathds{1}}

\newcommand{\paramANN}{\mathcal{P}}

\newcommand{\be}{\begin{eqnarray*}}
\newcommand{\ee}{\end{eqnarray*}}
\newcommand{\ben}{\begin{eqnarray}}
\newcommand{\een}{\end{eqnarray}}
\newcommand{\bi}{\begin{itemize}}
\newcommand{\ei}{\end{itemize}}

\newcommand{\SmallSum}[2]{ {\textstyle\sum\limits_{#1}^{#2}}}

\DeclarePairedDelimiter{\pr}{(}{)}
\DeclarePairedDelimiter{\cu}{\{}{\}}
\DeclarePairedDelimiter{\br}{[}{]}

\DeclarePairedDelimiter{\abs}{\lvert}{\rvert}
\DeclarePairedDelimiter{\norm}{\lVert}{\rVert}
\DeclarePairedDelimiter{\ceil}{\lceil}{\rceil}

\NewDocumentCommand{\fnorm}{sO{}m}{%
  {\IfBooleanTF{#1}
    {\fnormhelp{\left|}{\right|}{#3}}
    {\fnormhelp{#2|}{#2|}{#3}}}
}
\makeatletter
\newcommand{\fnormhelp}[3]{\mathpalette\fnormhelp@i{{#1}{#2}{#3}}}
\newcommand{\fnormhelp@i}[2]{\fnormhelp@ii#1#2}
\newcommand{\fnormhelp@ii}[4]{%
  \sbox\z@{$\m@th#1#2#4#3$}%
  \sbox\tw@{$\m@th\|$}%
  \mathopen{\hbox to\wd\tw@{\hss\vrule height \ht\z@ depth \dp\z@ width .3\wd\tw@\hss}}%
  #4
  \mathclose{\hbox to\wd\tw@{\hss\vrule height \ht\z@ depth \dp\z@ width .3\wd\tw@\hss}}%
}

\newtheorem{theo}{Theorem}[section]

\newtheorem{lemma}[theo]{Lemma}
\newtheorem{propo}[theo]{Proposition}
\newtheorem{corollary}[theo]{Corollary}

\theoremstyle{definition}

\newtheorem{defi}[theo]{Definition}

\newtheorem{remark}[theo]{Remark}

\newenvironment{mproof}[1]{\noindent \textit{Proof of {#1}.}}{\hfill \qed}

\title{Walk-on-Spheres Monte Carlo and deep neural network approximations of elliptic PDEs with drift and killing}
\author{Konrad Kleinberg$^{1}$, Thomas Kruse$^{2}$ \bigskip \\
\small{$^1$ Department of Mathematics \& Informatics, University of Wuppertal,}
\vspace{-0.1cm}\\
\small{Germany; e-mail: \texttt{kleinberg}\textcircled{\texttt{a}}\texttt{uni-wuppertal.de}}\smallskip\\
\small{$^2$ Department of Mathematics \& Informatics, University of Wuppertal,}
\vspace{-0.1cm}\\
\small{Germany; e-mail: \texttt{tkruse}\textcircled{\texttt{a}}\texttt{uni-wuppertal.de}}
}

\begin{document}

\maketitle

\begin{abstract}
    In this paper we provide Monte Carlo and deep neural network approximations for stochastic representations of solutions to linear elliptic partial differential equations with constant diffusion, drift and killing. Building on the modified Walk-on-Spheres algorithm of Beznea et al.\ (arXiv:2209.01432), we introduce Monte Carlo estimators that explicitly incorporate sampled random times arising in the analyzed stochastic representations. We establish uniform error bounds for these estimators and show that, under suitable assumptions, a prescribed approximation accuracy is achieved with sample complexities growing at most polynomially in both the inverse accuracy and the problem dimension. Furthermore, we prove a deep neural network approximation result for the stochastic representations. Assuming suitable neural network representations of the boundary data and the distance function to the boundary, we use the constructed Monte Carlo to design deep neural networks that approximate the representation uniformly with a number of parameters growing at most polynomially in the inverse accuracy and the problem dimension. These results extend previous complexity analyses to a broader class of elliptic equations involving drift and killing.
\end{abstract}

\section{Introduction \& Problem formulation}
In this work, we consider the approximation of functions $w \colon D \to \R$, 
\begin{equation} \label{eq:expectation}
    w\pr{x} = \E \br*{ h_0\pr{\tau^x, x + B_{\tau^x}} } + \E \br*{ \int_0^{\tau^x} h_1 \pr{s, x + B_s}  \dx s}, \qquad x \in D,
\end{equation}
where $d \in \N = \cu{1,2, \dots}$, $D \subseteq \R^d$ is open, bounded, and convex, $h_0, h_1 \colon [0, \infty) \times \ol{D} \to \R $ are bounded and Lipschitz continuous, $\pr{\Omega, \mathcal{F}, \mathbb{P}}$ is a probability space on which all appearing random variables are defined, $B = (B_t)_{t \in [0, \infty)} \colon \Omega \times [0, \infty) \to \R^d$ is a $d$-dimensional standard Brownian motion starting in $0$, and $\tau^x$, $x \in D$, is the first exit time of the process $x + B$ from $D$, i.e.,\ $\tau^x = \inf\cu{t \in [0, \infty) \colon x + B_t \in \partial D}$, $x \in D$.
The approximation of expectations of the form \eqref{eq:expectation} is particularly relevant in the numerical solution of partial differential equations (PDEs). Indeed, such representations appear naturally 
in stochastic representations of solutions of linear elliptic drift-diffusion PDEs on bounded domains of the following type:
\begin{equation}\label{eq:elliptic_pde}
\begin{split}
    \frac{1}{2}\tr(\sigma \sigma^\top D^2 u)+\langle b, \nabla u\rangle -ru +f &=0 \qquad \text{on } D, \\
    u\big |_{\partial D}&=g \qquad \text{on } \partial D,
\end{split}
\end{equation}
where
$u \colon \ol{D} \to \R$ is the unknown, $\sigma \in \R^{d\times d}$, $b \in \R^d$, $r \in [0, \infty)$, $f \colon D \to \R$, and $g\colon \partial D \to \R$ (see \cref{sec:representation} for a precise derivation of the relation between \eqref{eq:expectation} and \eqref{eq:elliptic_pde}). 
The goal of this work is to develop Monte Carlo schemes with error and complexity analyses as well as artificial neural network (ANN) approximations for expectations of the form \eqref{eq:expectation}, with a particular focus on high-dimensional situations.

Probabilistic approximation methods for elliptic boundary problems include in particular the so called Walk-on-Spheres algorithm, originally proposed in \cite{muller1956some} to approximate the solution of the Dirichlet problem (i.e., the special case of \eqref{eq:elliptic_pde} with $\sigma = \sqrt{2}\Id_{\R^d}$, $b = 0$, $r = 0$, $f = 0$) given by
\begin{equation} \label{eq:dirichlet}
    \begin{split}
        \Delta u &= 0 \qquad \text{on } D, \\
        u\big|_{\partial D} &= g \qquad \text{on } \partial D,
    \end{split}
\end{equation}
via Monte Carlo simulation. In particular, it allows to
efficiently approximate the exit point of Brownian motion from a bounded domain.
Conceptually the Walk-on-Spheres algorithm works as follows: At the current position $x \in D$ determine the largest sphere around $x$ contained in $D$. To obtain the next position, sample uniformly from this sphere.
The key observation here is that this new random position is indeed identically distributed to the location of first exit of a Brownian motion started in $x$ from said sphere.
To obtain an approximation for the exit point of Brownian motion from the entire domain, this procedure is repeated until the process is within a prescribed distance to the boundary of the domain. In particular, the number of steps required to reach that location is in general random.
In this original algorithm the spheres from which the next point of the process is sampled exhibit maximal radii, i.e., the radius of the sphere to be sampled from is given by the distance of the current position to the boundary of the domain.
Subsequent developments based on this algorithm cover methods for more general elliptic boundary value problems (see, e.g., \cite{sabelfeld1995integral,booth1981exact,kyprianou2018unbiased}).

Recently, there is also a growing interest in Monte Carlo approximations of elliptic PDEs that include a drift component. Indeed, in the work \cite{Sabelfeld2016random} the Walk-on-Spheres algorithm is extended to elliptic drift-diffusion equations by utilizing a von Mises-Fisher exit distribution on the successive spheres.
The article \cite{sabelfeld2020anewglobal} introduces a Monte Carlo method for self-adjoint elliptic PDEs with constant coefficients that estimates the solutions and its derivatives at multiple points from one ensemble of trajectories.
In \cite{sabelfeld2022global}, the authors extend this global approach to elliptic drift-diffusion equations.

To obtain global approximations for the drift-free Poisson problem (i.e., the special case of \eqref{eq:elliptic_pde} with $\sigma = \Id_{\R^{d}}$, $b = 0$, $r = 0$)
\begin{equation}
    \begin{split}\label{eq:poisson}
        \frac{1}{2} \Delta u + f &= 0 \qquad \text{on } D, \\
        u\big|_{\partial D} &= g \qquad \text{on } \partial D,
    \end{split}
\end{equation}
the works \cite{grohs2022deepelliptic} and \cite{beznea2022monte} combine Walk-on-Spheres based Monte Carlo estimators with suitable neural networks representations and approximations. In particular, under suitable assumptions these works 
construct ANNs that approximate the solution of \eqref{eq:poisson} with a complexity that only grows polynomially in both the inverse of the approximation accuracy\footnote{The approximation error is measured in $L^2(D)$ in \cite{grohs2022deepelliptic} and in $L^\infty(D)$in \cite{beznea2022monte}.} and the problem dimension, thereby mitigating the so-called curse of dimensionality. The general strategy of combining compositional and approximation properties of ANNs with Monte Carlo methods to overcome the curse of dimensionality was first developed for parabolic linear PDEs  (see, e.g., \cite{JentzenSalimovaWelti2018,BernerGrohsJentzen2018, GrohsWurstemberger2018}). It was subsequently extended to nonlinear equations (see, e.g., \cite{hutzenthaler2019proof, cioica2022deep, neufeld2023deep, ackermann2023deep,ackermann2024deep, neufeld2024rectified, neufeld2024multilevel})\footnote{See also \cite{beck2023overviewdeep} for a survey on deep-learning based approximation methods for PDEs.}. Moreover, this approach was used, for example, in the context of two-player zero-sum stochastic differential games (see, e.g., \cite{reisinger2020rectified}) and Markov decision processes (see, e.g., \cite{beck2025nonlinear} and \cite{jentzen2025dnnControl}).

Two features of the framework developed in \cite{beznea2022monte} are particularly important for our work.
First, it employs
a modified version of the Walk-on-Spheres process 
utilizing a specific class of modified distance functions as a substitute for the true distance function in order to determine the radii of the successive spheres from which the successive positions are sampled. 
Second, the framework admits a deterministic stopping criterion, independent of the starting point of the Walk-on-Spheres process, in place of the random number of steps used in the classical Walk-on-Spheres algorithm.
Both, the deterministic stopping rule and the modified distance functions play an important role in the construction of the ANN approximations.
Furthermore, 
the authors of \cite{beznea2022monte} work under exterior ball condition on $D$ rather than assuming convexity of $D$.

Building on this framework, we analyze in the present work Monte Carlo estimators for expectations of the form \eqref{eq:expectation} and consequently of solutions of elliptic drift-diffusion equations of the form \eqref{eq:elliptic_pde}. 
More precisely, we consider Monte Carlo estimators
$\bar{w} \colon D \times \Omega \to \R$ satisfying for $x \in D$ that
\begin{equation}
    \begin{split} \label{eq:estimator_intro}
        \bar{w}\pr{x} &= \frac{1}{n} \sum_{i = 1}^{n} \bigg[ h_0\pr{\bar{\tau}^{x, i}_M, \bar{X}^{x, i}_M} + \frac{1}{d} \SmallSum{k = 1}{M} \mdist\pr{ \bar{X}^{x, i}_{k - 1} }^2 h_1 \pr*{ \bar{\tau}^{x, i}_{k - 1} + \mdist\pr{ \bar{X}^{x, i}_{k - 1} }^2 V_k^{i}, \bar{X}^{x, i}_{k - 1} + \mdist\pr{ \bar{X}^{x, i}_{k - 1} } Y_k^{i} } \bigg].
    \end{split}
\end{equation}
Here, the number $n \in \N$ denotes the number of sample paths of the Walk-on-Spheres process,
the number $M \in \N$ denotes the deterministic number of steps the employed Walk-on-Spheres process undertakes,
the function $\mdist \colon \ol{D} \to \R$ determines the radii of the successive spheres in the generation of the Walk-on-Spheres process and is to be understood as a suitable substitute for the true distance to the boundary of $D$.
For the sample with index $i \in \cu{1,\dots, n}$ the random variables $\bar{X}_0^{x, i}, \dots, \bar{X}_{M}^{x, i}$ denote the $M$ steps of the Walk-on-Spheres process with start in $\bar{X}^{x, i}_{0} = x \in D$. 
The random variable $\bar{\tau}^{x, i}_{k}$ correspond to the physical time a Brownian motion started in the position $\bar{X}^{x, i}_{k - 1}$ needs to hit the boundary of the sphere with radius $\mdist \pr{\bar{X}^{x, i}_{k - 1}}$ around $\bar{X}^{x, i}_{k - 1}$.
The random variable $(V_{k}^{i}, Y_{k}^{i})$ correspond to the time and space occupation of a Brownian motion with start in $0$ inside the unit sphere.

The estimator in \eqref{eq:estimator_intro} extends the estimator introduced in \cite{beznea2022monte} by incorporating the sampled random times $\bar \tau^{x,i}_M$, $x\in D$, $M\in \N$, $i\in \{1,\ldots,n\}$. This additional time dependency appears in stochastic representations of solutions of \eqref{eq:elliptic_pde} as a compensation for the drift $b$ and killing $r$ components (see \cref{prop:girsanov_feynman_kac,rem:FK_implies_GFK} and below).

We next briefly describe the two main contributions of this article. To characterize appearing complexities (e.g., sample complexity, number of parameters describing an ANN), we adopt the following terminology: A real number $c \in \R$ is called an absolute constant, if $c$ does not depend on any variables or parameters in the problem. Moreover, for quantities $Q, q_1, \dots, q_n \in \R$, where $n \in \N$, we say that $Q$ depends only on $q_1, \dots, q_n$ if $Q$ is an explicit function of the variables $q_1, \dots, q_n$ and $Q$ has no other dependencies. In this case, this dependency is at most polynomial if there exists an absolute constant $c \in [0, \infty)$ such that $Q \le c (1 + \abs{q_1} + \dots + \abs{q_n})^c$.

First, we establish an error analysis for Monte Carlo approximations \eqref{eq:estimator_intro} for expectations of the form \eqref{eq:expectation} in the $L^{\infty}$-sense (see \cref{thm:exp_bound} and \cref{cor:l_infty} below). 
In particular, we demonstrate that for a prescribed approximation accuracy $\eps \in \pr{0, 1}$ the estimator \eqref{eq:estimator_intro} achieves $\E \br{\sup_{x \in D} \abs{w(x) - \bar{w}(x)} } \le \eps$ under suitable assumptions on the modified distance function $\mdist$ as well as with suitable choices of $n, M\in \N$.
More specifically, $n$ and $M$ can be chosen such that $\max \cu{n, M} \le c  \eps^{-q}$, where $q \in [0, \infty)$ is an absolute constant and $c \in [0, \infty)$ is a constant that depends only on $\sup_{(t, x) \in [0, \infty) \times \ol{D}} \abs{h_i(t,x)}$, $i \in \cu{1,2}$, the Lipschitz constants of $h_i$, $i \in \cu{1,2}$, the Lipschitz constant as well as a nondegeneracy constant\footnote{For details, see \cref{def:beta_eps_distance} and \cref{cor:l_infty} below} of the modified distance $\mdist$, the diameter $\diam \pr{D}$ of the domain, and the dimension $d$ and this dependency is at most polynomial.
In particular, imposing that these problem features grow at most polynomially with respect to the dimension $d$ the Monte Carlo scheme indeed overcomes the curse of dimensionality.

Second, we establish in \cref{thm:ANN} an ANN approximation result for functions $w \colon D \to \R$, $w \pr{x} = \E \br{ h_0 \pr{\tau^{x}, x + B_{\tau^{x}}}}$ (i.e., \eqref{eq:expectation} in the special case $h_1 \equiv 0$). 
In particular,  under the assumption that $h_0$ can be represented by an ANN and the distance to the boundary $\dist(\cdot, \partial D)$ can be suitably approximated by an ANN then $w$ can be approximated in the $L^{\infty}$-sense with accuracy $\eps \in \pr{0,1}$ by an ANN $\Psi$ such that the number of parameters of $\Psi$ is bounded by $\bar{c} \eps^{-\bar{q}}$, where $\bar{q} \in [0, \infty)$ is an absolute constant and $\bar{c} \in [0, \infty)$ depends only on the number of parameters, the Lipschitz constant, and the least upper bound on $[0, \infty) \times \ol{D}$ of the ANN representing $h_0$, the number of parameters and a uniform Lipschitz constant of the ANNs approximating the distance to the boundary, $\diam \pr{D}$, and $d$ and this dependendy is at most polynomial. In particular, under similar assumptions as for the Monte Carlo estimator, a curse of dimensionality relief is obtained (see \cref{rem:ANN_discussion} below).

The remainder of this article is structured as follows.
In \cref{sec:representation} below we establish the connection between functions of the form \eqref{eq:expectation} and solutions of the PDE \eqref{eq:elliptic_pde}.
In \cref{sec:modified_wos} we introduce the modified Walk-on-Spheres process and some of its important distributional properties and we present a Walk-on-Spheres-based representation of the expectation \eqref{eq:expectation}.
In \cref{sec:error_mc} we present the Monte Carlo estimator based on the Walk-on-Spheres process established in \cref{sec:modified_wos} and provide an error analysis.
In \cref{sec:ANN} we derive ANN approximation results for expectations of the form \eqref{eq:expectation} utilizing the error analysis established in \cref{sec:error_mc}.

\section{A Girsanov-based Feynman-Kac representation} \label{sec:representation}

In this section, we establish the connection between expectations of the form \eqref{eq:expectation} and classical solutions of the PDE \eqref{eq:elliptic_pde}
via suitable stochastic representations.

Unless otherwise specified, let $ d \in \N $,
$ \sigma \in \R^{d \times d} $, $ b \in \R^d $, $ r \in \R $,
$ D \subseteq \R^d $,
let $ f \colon \ol{D} \to \R $, $ g \colon \partial D \to \R $.
In the following remark, we recall sufficient conditions to establish existence and uniqueness of a classical solution of the PDE \eqref{eq:elliptic_pde} (see, e.g., \cite[Chapter~6, Theorem~6.13]{gilbarg2001elliptic}, \cite[Chapter~6, Theorem~2.4]{friedman1975sdeVol1}).
\begin{remark}
    Assume $\sigma$ is invertible, assume $r \ge 0$,
    assume that $D \subseteq \R^d$ is open, connected, bounded, and satisfies an exterior ball condition at every boundary point,
    assume that $f$ is uniformly $\alpha$-Hölder continuous, for some $\alpha \in \pr{0,1}$,
    and assume that $g$ is continuous on $\partial D$.
    Then there exists a unique $u \in C^0\pr{\ol{D}} \cap C^2\pr{D}$ that satisfies $\eqref{eq:elliptic_pde}$.
\end{remark}

The next \cref{prop:girsanov_feynman_kac} gives rise to a stochastic representation of the solution of \eqref{eq:elliptic_pde} that is the base for the Monte Carlo estimates employed in subsequent section.
We fix some more notation. Unless otherwise specified, 
let $(\Omega, \cF, \mathbb{P})$ be a probability space,
let $B\colon [0,\infty)\times \Omega \to \R^d$ be a standard Brownian motion,
and let $(\cF_t)_{t\in [0,\infty)}$ be the completed filtration generated by $B$. 
For every $x\in \ol{D}$ let 
\begin{equation}
    \tau^x=\inf\{t\in [0,\infty) \colon x+ \sigma B_t\in \partial D\}.
\end{equation}

\begin{propo}\label{prop:girsanov_feynman_kac}
    Assume that $D$ is open, bounded, and connected, that $f$ and $g$ are continuous,
    that $\sigma$ is invertible, and that $r\ge -\frac{1}{2}\|\sigma^{-1}b\|^2$.
    Let $u\in C^0\pr{\overline D} \cap C^2 \pr{D}$ be a solution of \eqref{eq:elliptic_pde}. Then it holds for all $x\in D$ that
    \begin{multline}\label{eq:girsanov_feynman_kac}
    u(x)=\E\biggl[e^{-(r+\frac{1}{2}\|\sigma^{-1}b\|^2)\tau^x+\langle \sigma^{-1}b,B_{\tau^x}\rangle}g(x+\sigma B_{\tau^x})\\+\int_0^{\tau^x}e^{-(r+\frac{1}{2}\|\sigma^{-1}b\|^2)s+\langle \sigma^{-1}b,B_{s}\rangle}f(x+\sigma B_s) \dx s\biggr].
\end{multline}
\end{propo}

\begin{proof}
    Let $x\in D$. Note that It\^o's formula implies for all $t\in [0,\infty)$ that
    \begin{equation}
        \begin{split}
            \hspace{-0.28cm} \dx e^{ - \pr{r + \frac{1}{2} \norm{\sigma^{-1}b}^2 } (t \wedge \tau^x) + \langle \sigma^{-1}b, B_{t \wedge \tau^x}  \rangle } = \ind_{\cu{t < \tau^x}} e^{ - \pr{r + \frac{1}{2} \norm{\sigma^{-1}b}^2 }t + \langle \sigma^{-1}b, B_{t}  \rangle } \pr*{-r \dx t + \langle \sigma^{-1}b, \dx B_t \rangle}
        \end{split}
    \end{equation}
    and
    \begin{equation}
        \begin{split}
            \dx u\pr{x + \sigma B_{t \wedge \tau^x}} &= \ind_{\cu{t < \tau^x}} \pr[\big]{ \langle \nabla u(x+\sigma B_{t}),\sigma \dx B_t\rangle + \frac{1}{2}\tr(\sigma \sigma^\top D^2u(x+\sigma B_{t})) \dx t }.
        \end{split}
    \end{equation}
    Combining this, the fact that for all $t \in [0, \infty)$ it holds that $B_{t \wedge \tau^x} \ind_{\cu{t < \tau^x}} =  B_t \ind_{\cu{t < \tau^x}}$, and the fact that $u$ solves \eqref{eq:elliptic_pde} with the product rule entails for all $t \in [0, \infty)$ that
    \begin{equation}
        \begin{split}
            &\dx \pr*{ e^{ - \pr{r + \frac{1}{2} \norm{\sigma^{-1}b}^2} \pr{t \wedge \tau^x} + \langle \sigma^{-1} b, B_{t \wedge \tau^x} \rangle  } u \pr{ x + \sigma B_{t\wedge \tau^x} } }\\
            &= \ind_{\cu{t < \tau^x}} e^{ - \pr{r + \frac{1}{2} \norm{\sigma^{-1}b}^2 }t + \langle \sigma^{-1}b, B_{t}  \rangle } \Big( \langle \nabla u \pr{x + \sigma B_t}, \sigma \dx B_t \rangle - f\pr{x + \sigma B_t} \dx t\\
            &\quad- \langle b, \nabla u \pr{x + \sigma B_t} \rangle \dx t + ru \pr{x + \sigma B_t} \dx t - r u\pr{x + \sigma B_t} \dx t + \langle b, \nabla u \pr{x + \sigma B_t }\rangle \dx t  \\
            &\quad + u \pr{x + \sigma B_t} \langle \sigma^{-1}b, \dx B_t \rangle \Big) \\
            &= \ind_{\cu{t < \tau^x}} e^{ - \pr{r + \frac{1}{2} \norm{\sigma^{-1}b}^2 }t + \langle \sigma^{-1}b, B_{t}  \rangle } \Big( \langle \sigma^\top \nabla u \pr{x + \sigma B_t} + \sigma^{-1} b u \pr{x + \sigma B_t}, \dx B_t \rangle  \\
            &\quad- f\pr{x + \sigma B_t} \dx t  \Big).
        \end{split}
    \end{equation}
    Written as integrals it holds for all $t \in [0, \infty)$ that
    \begin{equation}
        \begin{split}
            &e^{ - \pr{r + \frac{1}{2} \norm{\sigma^{-1}b}^2} \pr{t \wedge \tau^x} + \langle \sigma^{-1} b, B_{t \wedge \tau^x} \rangle  } u \pr{ x + \sigma B_{t\wedge \tau^x} } - u(x) \\
            &= -\int_0^{t \wedge \tau^x} e^{ - \pr{r + \frac{1}{2} \norm{\sigma^{-1}b}^2} s + \langle \sigma^{-1}b, B_s \rangle } f(x + \sigma B_s) \dx s \\
            &\quad+ \int_0^{t \wedge \tau^x} e^{ - \pr{r + \frac{1}{2} \norm{\sigma^{-1}b}^2} s + \langle \sigma^{-1}b, B_s \rangle } \langle u\pr{x + \sigma B_s} \sigma^{-1} b + \sigma^{\top} \nabla u \pr{x + \sigma B_{s}}, \dx B_s \rangle.
        \end{split}
    \end{equation}
    Next, let $\tau_k$, $k\in \N$, be a localizing sequence of stopping times for the local martingale $\int_{0}^{\cdot \wedge \tau^x} e^{-(r+\frac{1}{2}\|\sigma^{-1}b\|^2)t+\langle \sigma^{-1}b,B_{t}\rangle}\langle u(x+\sigma B_{t}) \sigma^{-1}b+\sigma^\top\nabla u(x+\sigma B_{t}), \dx B_{t}\rangle$.
    This ensures for all $t \in [0, \infty)$, $k\in \N$ that
    \begin{multline}
        u(x) = \E\biggl[ e^{-(r+\frac{1}{2}\|\sigma^{-1}b\|^2)(t \wedge \tau_k \wedge \tau^x)+\langle \sigma^{-1}b,B_{t \wedge \tau_k \wedge \tau^x}\rangle} u(x+\sigma B_{t \wedge \tau_k \wedge \tau^x})\\
        +\int_0^{t \wedge \tau_k \wedge \tau^x} \hspace{-1cm} e^{-(r+\frac{1}{2}\|\sigma^{-1}b\|^2)s+\langle \sigma^{-1}b,B_{s}\rangle} f(x+\sigma B_{s}) \dx s \biggr].
    \end{multline}
    This, the dominated convergence theorem, and the fact that  $u\big |_{\partial D}=g$ on $\partial D$ demonstrate that
    \begin{equation}
        \begin{split}
            u(x)&=\lim_{\substack{t \to \infty \\ k\to \infty}} \E\biggl[ e^{-(r+\frac{1}{2}\|\sigma^{-1}b\|^2)(t \wedge \tau_k \wedge \tau^x)+\langle \sigma^{-1}b,B_{t \wedge \tau_k \wedge \tau^x}\rangle} u(x+\sigma B_{t \wedge\tau_k \wedge \tau^x})\\
            &\quad +\int_0^{t \wedge \tau_k \wedge \tau^x} \hspace{-1cm} e^{-(r+\frac{1}{2}\|\sigma^{-1}b\|^2)s+\langle \sigma^{-1}b,B_{s}\rangle} f(x+\sigma B_{s})\dx s \biggr]\\
            &=\E\biggl[ e^{-(r+\frac{1}{2}\|\sigma^{-1}b\|^2) \tau^x+\langle \sigma^{-1}b,B_{\tau^x}\rangle} u(x+\sigma B_{\tau^x})\\
            &\quad +\int_0^{\tau^x}e^{-(r+\frac{1}{2}\|\sigma^{-1}b\|^2)s+\langle \sigma^{-1}b,B_{s}\rangle} f(x+\sigma B_{s}) \dx s \biggr] \\
            &=\E\biggl[ e^{-(r+\frac{1}{2}\|\sigma^{-1}b\|^2) \tau^x+\langle \sigma^{-1}b,B_{\tau^x}\rangle} g(x+\sigma B_{\tau^x})\\
            &\quad +\int_0^{\tau^x}e^{-(r+\frac{1}{2}\|\sigma^{-1}b\|^2)s+\langle \sigma^{-1}b,B_{s}\rangle} f(x+\sigma B_{s}) \dx s \biggr].  
        \end{split}
    \end{equation}
    This completes the proof.
\end{proof}

In the following remark, we informally derive another stochastic representation of the solution $u$ of \eqref{eq:elliptic_pde} based on the Feynman-Kac representation and establish a connection to the representation \eqref{eq:girsanov_feynman_kac} as well as reason why the representation \eqref{eq:girsanov_feynman_kac} is preferable for our work.

\begin{remark} \label{rem:FK_implies_GFK}
    Let $W\colon [0,\infty)\times \Omega \to \R^d$ be a standard Brownian motion on $\Omega$ under the probability measure $Q$. For every $x\in \R^d$ let $X^x\colon [0,\infty)\times \Omega \to \R^d$ and $\tilde \tau^x\colon \Omega \to [0,\infty)$ satisfy $X^x_t =x+bt+\sigma W_t$, $t\in [0,\infty)$, and $\tilde{\tau}^{x} =\inf\{t\in [0,\infty)|X^x_t\in \partial D\}$.
    Then the Feynman-Kac formula yields under appropriate assumptions for every $x\in \R^d$ that
    \begin{equation}\label{eq:feynman_kac}
        u(x)=E^{Q}\left[e^{-r\tilde \tau^x}g(X_{\tilde \tau^x}^x)+\int_0^{\tilde \tau^x}e^{-rs}f(X_s^x) \dx s\right].
    \end{equation}
    In \cref{sec:modified_wos} and following, we employ the Walk-on-Spheres algorithm which is based on the isotropy of Brownian motion.
    For this reason, representation \eqref{eq:girsanov_feynman_kac}, which employs exit times of $x+\sigma B$, is more convenient than \eqref{eq:feynman_kac} which is based on the non-isotropic process $X^x$.

    Note that \eqref{eq:girsanov_feynman_kac} can be derived from \eqref{eq:feynman_kac} via Girsanov's theorem (see, e.g., \cite[Appendix~A.1]{sawhney2022gridfree} for a similar derivation).
    To this end, let $ K \colon [0,\infty)\times \Omega \to \R^d$, $L\colon [0,\infty)\times \Omega \to \R$, and $\tau_K^x \colon \Omega \to \br{0,\infty}$, $x \in \R^d$, satisfy 
    \begin{equation}
        \begin{split}
            K_t &= W_t+\sigma^{-1}b t, \qquad t\in [0,\infty),\\
            L_t &= e^{-\langle \sigma^{-1}b,W_t\rangle-\frac{1}{2}\|\sigma^{-1}b\|^2 t},  \qquad t\in [0,\infty),\\
            \tau_{K}^{x} &= \inf \cu{ t \in [0, \infty) \colon x + \sigma K_t \in \partial D}, \qquad x \in \R^d.
        \end{split}
    \end{equation}
    Note that for all $t\in [0,\infty)$, $x\in \R^d$ we have $X_{t}^{x} = x + \sigma K_{t}$, $L_t = e^{-\langle \sigma^{-1}b,K_t\rangle+\frac{1}{2}\|\sigma^{-1}b\|^2 t}$, and $\tilde \tau^x=\tau_{K}^{x}$.
    Then, Girsanov's theorem (see, e.g., \cite[Section~3.5, Theorem~5.1]{karatzas1991brownian}) implies for every $T\in [0,\infty)$ that $K$ is a Brownian motion on $(\Omega, \cF_T, (\cF_t)_{t\in [0,T]},P_T)$ with $\frac{ \dx P_T}{\dx Q}=L_T$. 
    By a suitable limiting procedure we obtain a measure $P_\infty$ such that $K$ is a Brownian motion under $P_\infty$ and it follows from \eqref{eq:feynman_kac} that for all $x\in \R^d$ it holds that
    \begin{multline}\label{eq:feynman_kac2}
        u(x)=E^{P_\infty}\biggl[e^{-(r+\frac{1}{2}\|\sigma^{-1}b\|^2)\tau_{K}^{x}+\langle \sigma^{-1}b,K_{\tau_{K}^{x}}\rangle}g(x+\sigma K_{\tau_{K}^{x}})\\
        +\int_0^{\tau_{K}^{x}}e^{-(r+\frac{1}{2}\|\sigma^{-1}b\|^2)s+\langle \sigma^{-1}b,K_{s}\rangle}f(x+ \sigma K_{s}) \dx s\biggr].
    \end{multline}
    This exactly corresponds to \eqref{eq:girsanov_feynman_kac}.
\end{remark}

\begin{remark}\label{rem:remove_sigma}
    Under the assumption that $\sigma$ is invertible one can transform the elliptic PDE \eqref{eq:elliptic_pde} to one with isotropic Laplacian. Indeed, let $v\in C^2(\sigma^{-1} D)$ satisfy
    \begin{equation}
    \begin{split}
        \frac{1}{2}\Delta v+\langle \sigma^{-1}b, \nabla v\rangle -rv +f(\sigma \cdot) &=0 \qquad \text{on } \sigma^{-1}D \\
    v\big |_{\partial (\sigma^{-1}D)}&=g(\sigma \cdot) \qquad \text{on } \partial (\sigma^{-1} D).
    \end{split}
    \end{equation}
    Then the function $u(x)=v(\sigma^{-1}x)$, $x\in D$, satisfies $u \in C^2(D)$ and \eqref{eq:elliptic_pde}. For this reason we focus attention on the case $\sigma=\Id$ in the remainder of this article.
    A proof for this transformation in the case $b = 0$, $r = 0$ is given in \cite[Lemma~2.1]{beznea2022monte}.
\end{remark}

\begin{remark} \label{rem:connection_pde_expectation}
    In the situation of \cref{prop:girsanov_feynman_kac} with $\sigma = \Id$ let $h_0 \colon [0, \infty) \times \partial D \to \R$, $h_1 \colon [0, \infty) \times \ol{D} \to \R$ satisfy for all $t \in [0, \infty)$, $x \in \partial D$, $y \in \ol{D}$ that
    \begin{equation}
        \begin{split} \label{eq:connection_pde_expectation_h_i}
            h_{0} \pr{t, x} = e^{ - \pr{r + \frac{1}{2} \norm{b}^2 }t + \langle  b, x \rangle } g\pr{x} \qquad \text{and} \qquad h_{1} \pr{t, y} = e^{ - \pr{r + \frac{1}{2} \norm{b}^2 }t + \langle b, y \rangle } f\pr{y}.
        \end{split}
    \end{equation}
   Then it holds for all $x \in D$ that 
    \begin{equation}
        u\pr{x} = e^{- \langle b, x \rangle} \pr*{ \E \br*{ h_0\pr{\tau^{x}, x + B_{\tau^{x}}} } +  \E \br*{ \int_{0}^{\tau^{x}} h_1\pr{s, x + B_s} \dx s  }  }.
    \end{equation}
    Under the additional assumption that $g$ is extendable to $\ol{D}$ and that both $g$ and $f$ are Lipschitz on $\ol{D}$, it holds that $h_0$ as in \eqref{eq:connection_pde_expectation_h_i} is well-defined on $[0, \infty) \times \ol{D}$ and $h_0$ and $h_1$ are bounded and Lipschitz continuous. Hence, the relation between the PDE solution $u$ and the expectation $w$ as in \eqref{eq:expectation} is established as $u(x) = e^{- \langle b, x \rangle} w\pr{x} $ for $x \in D$. We remark that this transformation is also used in \cite[Section 2]{Sabelfeld2016random}.
\end{remark}

\section{A modified Walk-on-Spheres algorithm} \label{sec:modified_wos}

In this section, we introduce the modified Walk-on-Spheres algorithm as proposed in \cite{beznea2022monte} and demonstrate some relevant distributional properties.
In particular, we present in \cref{lem:decay} estimates on the tail probability of the number of steps the Walk-on-Spheres process requires to reach the boundary of the domain within a certain distance.
Moreover, we establish in \cref{lem:WoS_representation_solution} a representation of functions of the form \eqref{eq:expectation} based on the established Walk-on-Spheres process which subsequently motivates the Monte Carlo estimator considered in \cref{sec:error_mc}.

The following lemma demonstrates a consequence of the rotational invariance of Brownian motion.

\begin{lemma}\label{lem:time_exit_BM_independence}
    Let $\tau = \inf\cu{t \in [0, \infty) \colon \norm{B_t} \ge 1}$. Then it holds that $\tau$ and $B_\tau$ are independent.
\end{lemma}

\begin{proof}
    Throughout this proof let $S = \cu{ y \in \R^d \colon \norm{y} = 1} $ and let $\mu \colon \cB\pr{S} \to [0,1]$ denote the uniform distribution on $(S, \cB \pr{S})$.
    First, observe that $\tau$ and $B_\tau$ are measurable path functionals of the path of B.
    Next, note that the fact that for all $R \in SO(d)$ it holds that $B$ and $RB$ are identically distributed
    and the fact that for all $R \in SO(d)$, $x \in \R^d$ it holds that $\norm{Rx} = \norm{x}$ 
    yield for all $R \in SO(d)$ that $(\tau, B_\tau)$ and $(\tau, R B_\tau)$ are identically distributed.
    Moreover, observe that the fact that $\br{0, \infty}$ is a Polish space ensures that there exists a regular conditional distribution $\kappa \colon \br{0, \infty} \times \cB \pr{S} \to \br{0,1}$ of $B_\tau$ given $\tau$ (cf.,\ \cite[Section~8.3]{Klenke2014Probability}). 
    This implies for all $A \in \cB\pr{\br{0, \infty}}$, $C \in \cB\pr{S}$ that
    \begin{equation}
        \begin{split} \label{eq:time_exit_BM_independence_joint_distr}
            \mathbb{P}\br{ \tau \in A, B_\tau \in C } &= \E \br*{ \ind_{\cu{\tau \in A}} \ind_{\cu{B_\tau \in C}}  } = \E \br*{ \ind_{\cu{\tau \in A}} \E \br*{ \ind_{\cu{B_\tau \in C}} \mid \tau } } \\
            &= \int_\Omega \E \br*{ \ind_{\cu{B_\tau \in C}} \mid \tau } \ind_{\cu{\tau \in A}} \dx \mathbb{P} = \int_A  \kappa \pr{t, C} \pr{\mathbb{P} \circ \tau^{-1}} \pr{\dx t}.
        \end{split}
    \end{equation}
    Combining this
    with the fact that for all $R \in SO(d)$ it holds that $(\tau, B_\tau)$ and $(\tau, R B_\tau)$ are identically distributed
    demonstrates for all $A \in \cB\pr{\br{0, \infty}}$, $C \in \cB\pr{S}$, $R \in SO(d)$ that
    \begin{equation}
        \begin{split}
            \int_A \kappa\pr{t, C} \pr{\mathbb{P} \circ \tau^{-1}} \pr{\dx t} &= \mathbb{P}\br{ \tau \in A, B_\tau \in C } = \mathbb{P}\br{ \tau \in A, RB_\tau \in C } \\
            &= \int_A \kappa \pr{t, R^\top C} \pr{\mathbb{P} \circ \tau^{-1}} \pr{\dx t}.
        \end{split}
    \end{equation}
    This implies that for $\mathbb{P} \circ \tau^{-1}$-a.a. $t \in [0, \infty)$ it holds that $\kappa\pr{t, \cdot}$ is a rotational invariant probability measure on $S$.
    This and
    the fact that $\mu$ is the unique rotationally invariant probability measure on $S$ (cf., \cite[Theorem~3.4]{Mattila1995Geometry})
    imply that for $\mathbb{P} \circ \tau^{-1}$-a.a. $t \in [0, \infty)$ it holds that $\kappa\pr{t, \cdot} = \mu$.
    This and \eqref{eq:time_exit_BM_independence_joint_distr} yield for all $A \in \cB\pr{\br{0, \infty}}$, $C \in \cB\pr{S}$ that
    \begin{equation}
        \mathbb{P}\br{ \tau \in A, B_\tau \in C } = \mathbb{P} \br{\tau \in A} \mu\pr{C}.
    \end{equation}
    This shows that $\tau$ and $B_\tau$ are independent.
    The proof of \cref{lem:time_exit_BM_independence} is thus complete.
\end{proof}

\noindent
Throughout the remainder of this section, assume that $D$ is open and bounded. Let $r\colon \R^d \to [0,\infty)$ satisfy for all $x\in \R^d$ that
\begin{equation}
    r(x)=\inf\{\|x-y\| \colon y\in \partial D\}.
\end{equation}

\noindent
The following lemma is a modification of \cite[Lemma~3.2]{grohs2022deepelliptic}. The geometric proof given there can be adapted to yield the following statement. Recall for all $x \in \ol{D}$ that $\tau^{x} = \inf \cu{ t \in [0, \infty) \colon x + B_t \in \partial D }$.

\begin{lemma} \label{lem:expected_exit_time_UB}
    Assume that $D$ is convex. Then it holds for all $x \in D$ that
    \begin{equation}
        \E \br*{\tau^x} \le r\pr{x} \diam \pr{D}.
    \end{equation}
\end{lemma}

\noindent
The next definition is \cite[Definition~2.7]{beznea2022monte} and describes the class of modified distance functions, which subsequently determines the radii of the successive spheres on which the Walk-on-Spheres process is defined. 

\begin{defi} \label{def:beta_eps_distance}
    Let $\beta \in (0, 1]$, $\eps \in [0,\infty)$,
    let $\tilde{r} \colon D \to [0, \infty)$ be Lipschitz continuous,
    and assume for all $x,y \in D$ with $r\pr{x} \ge \eps$ that $\tilde{r}\pr{y} \le r\pr{y}$ and $\tilde{r}(x) \ge \beta r\pr{x}$. 
    Then $\tilde{r}$ is called $\pr{\beta, \eps}$-distance on $D$.
\end{defi}

\noindent
With this notion of modified distance functions, we now introduce the Walk-on-Spheres process.
Let $\beta \in (0, 1]$, $\delta \in [0,\infty)$, let $\mdist$ be a $\pr{\beta, \delta}$-distance on $D$.
Let $\tau_k^{x}\colon \Omega \to [0,\infty)$, $k\in \N_0$, satisfy for all $x\in \ol{D}$, $k\in \N$ that $\tau_0^{x}=0$
and
\begin{equation}
    \tau^{x}_{k}=\inf\cu{s\in [\tau^x_{k - 1},\infty) \colon \|B_s-B_{\tau^x_{k-1}}\|\ge \mdist(x+B_{\tau^x_{k-1}})},
\end{equation}
let $Z^j_k\colon \Omega \to \R^d$, $j,k\in \N$, be standard normally distributed, let $\xi^j_k\colon \Omega \to \R^d$, $j,k\in \N$, satisfy for all $j,k\in \N$ that $\xi^j_k=\|Z^j_k\|^{-1}Z^j_k$, let $\bar X^{x,j}_k\colon \Omega \to \ol{D}$, $x\in \ol{D}$, $j\in \N$, $k\in \N_0$, satisfy for all $x\in \ol{D}$, $j,k\in \N$ that $\bar X^{x,j}_0=x$ and
\begin{equation}
    \bar X^{x,j}_k=\bar X^{x,j}_{k-1}+\mdist(\bar X^{x,j}_{k-1})\xi^j_k,
\end{equation}
let $\zeta\colon \Omega \to [0,\infty)$ satisfy $\zeta=\inf\{t\in [0,\infty) \colon \|B_t\|\ge 1\}$, let $\rho_k^j$, $k,j\in \N$, be i.i.d.\ copies of $\zeta$, assume that $\rho_k^j$, $Z^j_k$, $j,k\in \N$ are independent, let $\bar \tau_k^{x,j}\colon \Omega \to [0,\infty)$, $j\in \N$, $k\in \N_0$ satisfy for all $x\in \ol{D}$, $j,k\in \N$ that $\bar{\tau}_{0}^{x,j}=0$ and
\begin{equation}
    \bar \tau^{x,j}_{k}=\bar \tau^{x,j}_{k-1} + \mdist^2(\bar X_{k-1}^{x,j})\rho_k^j.
\end{equation}

\begin{lemma} \label{lem:distribution_WoS_BM}
    For all $x\in \ol{D}$, $k\in \N_0$, $j\in \N$ it holds that $(\bar \tau^{x,j}_k,\bar X^{x,j}_k)$ and $(\tau^x_k, x+B_{\tau^x_k})$ are identically distributed.
\end{lemma}
\begin{proof}
    First, note that $(\bar \tau^{x,j}_0,\bar X^{x,j}_0) = (0, x) = (\tau^x_0, x+B_{\tau^x_0})$.
    Next, observe that by assumption it holds that $\bar{ \tau}^{x,j}_1= \mdist^{2}(x)\rho^j_1$ and $\bar{X}^{x,j}_1=x+\mdist(x)\xi^j_1$ are independent.
    Moreover, note that \cref{lem:time_exit_BM_independence} proves that $\tau_1^x$ and $x + B_{\tau_1^x}$ are independent.
    Combining this,
    the fact that $\tau_1^x$ and $\bar \tau^{x,j}_1$ are identically distributed,
    and the fact that $x+B_{\tau^x_1}$ and $\bar X^{x,j}_1$ are identically distributed
    yields that $\pr{\bar \tau^{x,j}_1, \bar X^{x,j}_1}$ and $\pr{\tau_1^x, x+B_{\tau^x_1}}$ are identically distributed.
    This, the strong Markov property (cf.,\ \cite[Chapter~2.6]{karatzas1991brownian}), and induction complete the proof of \cref{lem:distribution_WoS_BM}.
\end{proof}

\noindent For every $\eps \in (0,\infty)$ let
\begin{equation}
    D_\eps=\{x \in D \colon r(x) \ge \eps \}.
\end{equation}
For every $\eps \in (0,\infty)$, $x\in D$, $j\in \N$ let
\begin{equation}
    k_\eps^x = \inf\{k\in \N_0 \colon x+ B_{\tau^x_k}\notin D_\eps\},
\end{equation}
\begin{equation}
    \bar k_\eps^{x,j} = \inf\{k\in \N_0 \colon \bar X^{x,j}_{k}\notin D_\eps\},
\end{equation}
\begin{lemma} \label{lem:distribution_k_max}
    For all $\eps \in \pr{0,\infty}$, $x \in D$ it holds that $k_\eps^x$ and $\bar{k}^{x, 1}_\eps$ are identically distributed. 
\end{lemma}

\begin{proof}
    The strong Markov property (cf.,\ \cite[Section~2.6]{karatzas1991brownian}) and the assumption that $\rho_k^j, Z_k^j$, $j, k \in \N$ are independent ensure for all $x \in \ol{D}$ that $\pr{\bar{X}^{x, 1}_k}_{k \in \N_0}$ and $\pr{x + B_{\tau_k^{x}}}_{k \in \N_0}$ are identically distributed.
    This implies for all $\eps \in \pr{0, \infty}$, $x \in \ol{D}$ that $k_\eps^x$ and $\bar{k}^{x, 1}_\eps$ are identically distributed. %
    This completes the proof of \cref{lem:distribution_k_max}.
\end{proof}

\noindent
The combination of the next two lemmas quantifies the decay of the tail probabilities of $\bar{k}^{x, 1}_{\delta}$ the number of steps the Walk-on-Spheres algorithm needs to reach regions of the domain close to the boundary. To this end \cref{lem:mgf_exit_time_BM} provides bounds for the moment generating function of a Brownian motions first exit from a sphere and \cref{lem:decay} establishes the estimates of the tail probabilities of $\bar{k}^{x, 1}_{\delta}$.
The statements are modifications of \cite[Proposition~2.11]{beznea2022monte}. We provide proofs for the convenience of the reader.

\begin{lemma}\label{lem:mgf_exit_time_BM}
    Let $\mathfrak{r} \in \pr{0,\infty}$, $\gamma \in \pr{0, \frac{d}{\mathfrak{r}^2}}$
    and
    let $\tau = \inf \cu{t \in [0, \infty) \colon \norm{B_t} = \mathfrak{r}}$,
    Then it holds that
    \begin{equation} \label{eq:mgf_exit_time_BM_bounds}
        1 + \frac{\gamma \mathfrak{r}^2}{d} \le \E \br{e^{\gamma \tau}} \le \frac{1}{1 - \frac{\gamma \mathfrak{r}^2}{d}}.
    \end{equation}
\end{lemma}

\begin{proof}
    Throughout this proof let $v, w \colon \R^d \to \R$ satisfy for all $x \in \R^d$ that
    \begin{equation}
        v \pr{x} = \frac{1}{2} - \frac{\gamma \norm{x}^2}{2d} \qquad \text{and} \qquad w\pr{x} = 1 - \frac{\gamma \norm{x}^2}{d + \gamma \mathfrak{r}^2}.
    \end{equation}
    It\^o's formula yields for every $h \in \cu{v,w}$, $t \in [0,\infty)$ that
    \begin{equation} \label{eq:mgf_exit_time_BM_ito}
        e^{\gamma t} h\pr{B_t} = h\pr{0} + \int_0^t e^{\gamma s} \pr{\gamma h \pr{B_s} + \frac{1}{2} \Delta h \pr{B_s}} \dx s + \int_0^t e^{\gamma s} \langle \nabla h\pr{B_s}, \dx B_s \rangle.
    \end{equation}
    For every $h \in \cu{v,w}$ let $\tau_k^{h}$, $k \in \N$, be a localizing sequence of stopping times for the local martingale $\int_0^{\cdot} e^{\gamma s} \langle \nabla h\pr{B_s}, \dx B_s \rangle$ and for every $h \in \cu{v,w}$, $k \in \N$ let $T_k^{h} \colon \Omega \to \R$ satisfy $T_k^h = k \wedge \tau \wedge \tau_k^{h}$. 
    Combining the fact that for all $h \in \cu{v,w}$, $k \in \N$ it holds that $T_k^{h} \le k \wedge \tau$,
    the fact that for all $h \in \cu{v,w}$ it holds that $(\tau_k^h)_{k \in \N}$ localizes $\int_0^{\cdot} e^{\gamma s} \langle \nabla h \pr{B_s}, \dx B_s \rangle$,
    and the optional stopping theorem with \eqref{eq:mgf_exit_time_BM_ito} establishes for all $h \in \cu{v,w}$, $k \in \N$ that
    \begin{equation}
        \begin{split} \label{eq:mgf_exit_time_BM_stopped_exp}
            \E \br*{e^{\gamma T_k} h \pr{B_{T_k}}} = h\pr{0} + \E \br*{\int_0^{T_k} \!\!\!\!\!\! e^{\gamma s} \pr{\gamma h \pr{B_s} + \frac{1}{2} \Delta h \pr{B_s}} \dx s }.
        \end{split}
    \end{equation}
    Furthermore, note that the fact that for every $h \in \cu{v,w}$ it holds that $\mathbb{P} \pr{\lim_{k \to \infty} \tau_k^h = \infty} = 1$ ensures for every $h \in \cu{v,w}$ that $\mathbb{P} \pr{ \lim_{k \to \infty} T_k^h = \tau} = 1$.
    This and the fact that $ \mathbb{P} \pr{\tau < \infty} = 1$ (cf.\ e.g.,\ \cite[Section~4.2]{karatzas1991brownian}, \cite[Lemma~4.27]{muller2012monte}) imply for every $h \in \cu{v,w}$ that 
    \begin{equation} \label{eq:mgf_exit_time_BM_convergence}
        \mathbb{P} \pr*{ \lim_{k \to \infty} e^{\gamma T_k^h} h \pr{B_{T_k^h}} = e^{\gamma \tau} h\pr{B_\tau} } = 1.
    \end{equation}
    Moreover, observe that the facts that $\gamma \mathfrak{r}^2 < d$
    and that for all $k \in \N$ it holds that $T_k^v \le \tau$
    assure for all $k \in \N$ that $e^{\gamma T_k^v} v \pr{B_{T_k^v}} \ge 0$. 
    This,
    the fact that for all $x \in \R^d$ it holds that $\gamma v \pr{x} + \frac{1}{2} \Delta v\pr{x} = -\frac{\gamma^2 \norm{x}^2}{2d}\le 0$,
    the fact that $v \pr{B_\tau} = \frac{1}{2} - \frac{\gamma \mathfrak{r}^2}{2d} > 0$,
    the fact that for all $k \in \N$ it holds that $T_k^v \le \tau$,
    \eqref{eq:mgf_exit_time_BM_stopped_exp},
    \eqref{eq:mgf_exit_time_BM_convergence},
    and Fatou's lemma demonstrate that
    \begin{equation}
        \E \br{e^{\gamma \tau}} \pr*{\frac{1}{2} - \frac{\gamma \mathfrak{r}^2}{2d}} = \E \br*{e^{\gamma \tau} v \pr{B_\tau}} \le \liminf_{k \to \infty} \E \br*{ e^{\gamma T_k^v} v\pr{B_{T_k^v}} } \le v \pr{0} = \frac{1}{2}.
    \end{equation}
    Hence, it holds that
    \begin{equation} \label{eq:mgf_exit_time_BM_upper_bound}
        \E \br{e^{\gamma \tau}} \le \frac{1}{1 - \frac{\gamma \mathfrak{r}^2}{d}}.
    \end{equation}
    Next, note that the fact that for all $x \in \R^d$ it holds that $w\pr{x} \le 1$. 
    Combining this with the fact that it holds for all $k \in \N$ that $T_k^w \le \tau$ assures that for all $k \in \N$ it holds that $e^{\gamma T_k^w} w \pr{B_{T_k^w}} \le e^{\gamma \tau}$.
    This,
    the fact that for all $x \in \R^d$ with $\norm{x} \le \mathfrak{r}$ it holds that $\gamma w\pr{x} + \frac{1}{2} \Delta w \pr{x} = \frac{\gamma^2(\mathfrak{r}^2-\norm{x}^2)}{d+\gamma \mathfrak{r}^2} \ge 0$,
    the fact that $w \pr{B_\tau} = 1 - \frac{\gamma \mathfrak{r}^2}{d + \gamma \mathfrak{r}^2} > 0$,
    \eqref{eq:mgf_exit_time_BM_stopped_exp},
    \eqref{eq:mgf_exit_time_BM_convergence},
    \eqref{eq:mgf_exit_time_BM_upper_bound},
    and the dominated convergence theorem yield that
    \begin{equation} \label{eq:mgf_exit_time_BM_lower_bound}
        \E \br{e^{\gamma \tau}} \pr*{1 - \frac{\gamma \mathfrak{r}^2}{d + \gamma \mathfrak{r}^2}} = \E \br*{e^{\gamma \tau} w\pr{B_\tau}} = \lim_{k \to \infty} \E\br{ e^{\gamma T_k^w} w \pr{B_{T_k^w}} } \ge w \pr{0} = 1.
    \end{equation}
    This demonstrates
    \begin{equation}
        \E \br{e^{\gamma \tau}} \ge \frac{1}{1 - \frac{\gamma \mathfrak{r}^2}{d + \gamma \mathfrak{r}^2}} = \frac{d + \gamma \mathfrak{r}^2}{d} = 1 + \frac{\gamma \mathfrak{r}^2}{d}. 
    \end{equation}
    This and \eqref{eq:mgf_exit_time_BM_upper_bound} establish \eqref{eq:mgf_exit_time_BM_bounds}. 
    This completes the proof of \cref{lem:mgf_exit_time_BM}.
\end{proof}

\begin{lemma} \label{lem:decay} 
    Assume $\delta > 0$.
    Then it holds for all $M \in \N$, $x \in D$ that
    \begin{equation} \label{eq:decay}
        \mathbb{P} \pr*{\bar{k}^{x,1}_\delta \ge M} \le 2 \pr*{1 + \tfrac{\beta^2 \delta^2}{2\diam \pr{D}^2}}^{-M}. 
    \end{equation}
\end{lemma}

\begin{proof}
    Throughout this proof
    let $T_c \colon \Omega \to \br{0,\infty}$, $c \in \pr{0, \infty}$, satisfy for every $c \in \pr{0, \infty}$ that $T_c = \inf \cu{ t \in [0,\infty) \colon \norm{B_t} = c}$
    and let $\lambda\pr{\gamma} \in \R$, $\gamma \in \pr{0, \tfrac{d}{\diam \pr{D}^2}}$, satisfy for every $\gamma \in \pr{0, \tfrac{d}{\diam \pr{D}^2}}$ that $\lambda\pr{\gamma} = \log \pr{1 + \gamma \tfrac{\beta^2\delta^2}{d}}$.
    First, observe that
    the fact that for all $x \in D$, $y \in D_\delta$ it holds that $\mdist\pr{x} \le r\pr{x}$ and $\mdist \pr{y} \ge \beta r\pr{y} \ge \beta \delta$
    and \cref{lem:mgf_exit_time_BM} assure for every $\gamma \in \pr{0, \frac{d}{\diam \pr{D}^2}}$, $x \in D_\delta$ that
    \begin{equation} \label{eq:decay_exit_diam_upper}
        \E \br*{ e^{\gamma T_{\diam\pr{D}}} } \le \frac{d}{d - \gamma \diam \pr{D}^2}
    \end{equation}
    and
    \begin{equation} \label{eq:decay_exit_rx_lower}
        \E \br*{e^{\gamma T_{\mdist\pr{x}}}} \ge 1 + \frac{\gamma \mdist\pr{x}^2}{d} \ge 1 + \gamma \frac{\beta^2 \delta^2}{d} = e^{\lambda\pr{\gamma}}.
    \end{equation}
    Throughout the remainder of this proof let $K_\delta \colon D^{\N_0} \to \N_0 \cup \cu{+\infty}$ satisfy for every $x = \pr{x_0, x_1, \dots} \in D^{\N_0}$ that $K_\delta (x) = \inf \cu{n \in \N_0 \colon r\pr{x_n} < \delta}$. 
    Observe that $K_\delta$ is measurable.
    Moreover, note that for all $x \in D$, $j \in \N$ it holds that $K_\delta \pr{\bar{X}^{x, j}} = \bar{k}_\delta^{x, j}$.
    This,
    \cref{lem:distribution_k_max},
    the fact that $\mathbb{P} \pr{k_\delta^x < \infty} = 1$,
    \eqref{eq:decay_exit_rx_lower},
    the fact that for all $x \in D$ it holds that $T_{ \diam\pr{D} } \ge \tau^x_{k_\delta^x}$,
    the monotone convergence theorem,
    and the law of total expectation
    yield for all $\gamma \in \pr{0, \tfrac{d}{\diam\pr{D}^2}}$, $x \in D_\delta$ that
    \begin{equation}
        \begin{split} \label{eq:decay_conditioning_I}
            \tfrac{d}{d - \gamma \diam\pr{D}^2} & \ge \E \br*{ e^{\gamma T_{ \diam\pr{D} }} } \ge \E \br*{e^{\gamma \tau^x_{k_\delta^x} }} = \E \br*{e^{\gamma \bar{\tau}^{x, 1}_{\bar{k}_\delta^{x,1}} }}\\
            &= \E \br*{ \exp\pr*{\textstyle \gamma \sum_{i = 1}^{\bar{k}_\delta^{x,1}} \mdist \pr*{ \bar{X}^{x, 1}_{i - 1} }^2 \rho^1_i } } \\
            &= \lim_{n \to \infty} \E \br*{ \exp\pr*{\textstyle \gamma \sum_{i = 1}^{\bar{k}_\delta^{x,1} \wedge n} \mdist \pr*{ \bar{X}^{x, 1}_{i - 1} }^2 \rho^1_i } } \\
            &= \lim_{n \to \infty} \E \br*{ \E \br*{ \exp\pr*{\textstyle \gamma \sum_{i = 1}^{\bar{k}_\delta^{x,1} \wedge n} \mdist \pr*{ \bar{X}^{x, 1}_{i - 1} }^2 \rho^1_i } \; \Big| \; \bar{X}^{x, 1} }  }.
        \end{split}
    \end{equation}
    The fact that $K_\delta\pr{ \bar{X}^{x, 1} }$ is $\sigma \pr{\bar{X}^{x, 1}}$-measurable
    and the fact that $\bar{X}^{x,1}$ and $\pr{ \rho_n^1}_{n \in \N}$ are independent
    ensures for every $x \in D_\delta$, $n \in \N$, $\gamma \in \pr{0, \tfrac{d}{\diam \pr{D}^2}}$ that $\phi_{n, \gamma} \pr{\bar{X}^{x,1}}$, where
    \begin{equation}
        \begin{split} \label{eq:decay_conditional_expectation}
            \phi_{n, \gamma} \pr{y} &= \E \br*{ \exp\pr*{ \textstyle \gamma \sum_{i = 1}^{K_\delta\pr{y} \wedge n} \mdist\pr*{y_{i - 1}}^2 \rho_i^1 } } \\
            &= \int_\Omega \exp\pr*{ \textstyle \gamma \sum_{i = 1}^{K_\delta\pr{y} \wedge n} \mdist\pr*{y_{i - 1}}^2 \rho_i^1\pr{\omega} } \mathbb{P} \pr{\dx \omega}  , \quad y \in D^{\N_0},    
        \end{split}
    \end{equation}
    is (a version of) the conditional expectation $\E \br{ \exp\pr{\textstyle \gamma \sum_{i = 1}^{\bar{k}_\delta^{x,1} \wedge n} \mdist \pr{ \bar{X}^{x, 1}_{i - 1} }^2 \rho^1_i } \; | \; \bar{X}^{x, 1} } $.
    The fact that for all $y \in D^{\N_0}$, $j \in \cu{0,1,\dots, K_{\delta}\pr{y}}$ it holds that $y_{j - 1} \in D_\delta$,
    \eqref{eq:decay_exit_rx_lower},
    \eqref{eq:decay_conditional_expectation},
    and the assumption that $(\rho_n^1)_{n \in \N}$ are independent
    imply for all $y \in D^{\N_0}$, $n \in \N$, $\gamma \in \pr{0, \tfrac{d}{\diam\pr{D}^2}}$ that
    \begin{equation}
        \begin{split}
            \phi_{n, \gamma}(y) &= \!\! \prod_{i = 1}^{K_\delta \pr{y} \wedge n} \!\! \E \br*{ e^{\gamma \mdist \pr{y_{i - 1}}^2 \rho_i^{1}}} = \!\! \prod_{i = 1}^{K_\delta \pr{y} \wedge n} \!\! \E \br*{ e^{\gamma T_{\mdist\pr{y_{i - 1}}}} } \ge \!\!\prod_{i = 1}^{K_\delta \pr{y} \wedge n} \!\! e^{\lambda \pr{\gamma}} = e^{\lambda\pr{\gamma} \pr{ K_\delta \pr{y} \wedge n }}.
        \end{split}
    \end{equation}
    This,
    \eqref{eq:decay_conditioning_I},
    and the monotone convergence theorem
    ensure for all $x \in D_\delta$, $\gamma \in \pr{0, \tfrac{d}{\diam\pr{D}^2}}$ that
    \begin{equation}
        \begin{split}
            \hspace{-0.5cm}\tfrac{d}{d - \gamma \diam \pr{D}^2} &\ge \lim_{n \to \infty} \E \br*{ \E \br*{ \exp\pr*{\textstyle \gamma \sum_{i = 1}^{\bar{k}_\delta^{x,1} \wedge n} \mdist \pr*{ \bar{X}^{x, 1}_{i - 1} }^2 \rho^1_i } \; \Big| \; \bar{X}^{x, 1} }  } = \lim_{n \to \infty} \E \br*{ \phi_{n, \gamma}\pr{\bar{X}^{x, 1}} } \\
            &\ge \lim_{n \to \infty} \E \br*{ e^{\lambda\pr{\gamma} \pr*{K_\delta\pr{\bar{X}^{x,1}} \wedge n }} } = \E \br*{ e^{\lambda \pr{\gamma} K_\delta\pr{\bar{X}^{x, 1}}} } \\
            &= \E \br*{ e^{ \lambda\pr{\gamma} \bar{k}^{x,1}_\delta }} = \E \br*{ \pr*{ 1 + \gamma \tfrac{\beta^2 \delta^2}{d} }^{\bar{k}^{x,1}_\delta} }.
        \end{split}
    \end{equation}
    Choosing $\gamma = \tfrac{d}{2\diam\pr{D}^2}$ demonstrates for all $x \in D_\delta$ that
    \begin{equation}
         \E \br*{ \pr*{ 1 + \tfrac{\beta^2 \delta^2}{2 \diam \pr{D}^2} }^{\bar{k}^{x,1}_\delta} } \le 2.
    \end{equation}
    This and Markov's inequality yield for all $x \in D_\delta$, $M \in \N$ that
    \begin{equation}
        \mathbb{P} \pr{\bar{k}^{x, 1}_\delta \ge M} \le 2 \pr*{1 + \tfrac{\beta^2 \delta^2}{2\diam \pr{D}^2}}^{-M}.
    \end{equation}
    This proves \eqref{eq:decay}. This completes the proof of \cref{lem:decay}.
\end{proof}

\noindent
The next lemma provides a probability measure to calculate expectations of integrals of time dependent functions of Brownian motion based on the time-space occupation behavior of Brownian motion.

\begin{lemma} \label{lem:occ_measure}
    Let $U = \cu{ x \in \R^d \colon \norm{x} \le 1 }$, and
    let $\mu \colon \cB \pr{[0, \infty) \times U} \to [0, \infty)$ satisfy for all $A \in \cB\pr{[0, \infty) \times U}$ that $\mu\pr{A} = d \E \br*{\int_0^\zeta \ind_A \pr{s, B_s} \dx s}$. Then
    \begin{enumerate}[label = (\roman*)]
        \item \label{it:occ_measure_probability} it holds that $\mu$ is a probability measure and
        \item \label{it:occ_measure_integral} it holds for all measurable $f \colon [0, \infty) \times U \to \br{-\infty, \infty}$ where $f$ is $\mu$-integrable or satisfies $f \ge 0$ that 
        \begin{equation}
            \int_{[0, \infty) \times U} f\pr{p} \mu\pr{\dx p} = d \E \br*{ \int_0^\zeta f(s, B_s) \dx s }.
        \end{equation}
    \end{enumerate}
\end{lemma}
\begin{proof}
    Throughout this proof let $\nu \colon \cB \pr{[0, \infty) \times U} \ni A \mapsto \E \br*{\int_0^{\zeta} \ind_A \pr{s, B_s} \dx s} \in  [0, \infty]$.
    First, observe that \cite[e.g., Proposition~2.2.21 or 3.1.8]{port1978brownian} ensures that $\E \br{ \zeta } = \frac{1}{d}$. This proves that $\nu$ is a finite measure with $\nu \pr{[0, \infty) \times U} = \E \br*{ \zeta } = \frac{1}{d}$.
    This and the fact that for all $A \in \cB \pr{[0, \infty) \times U}$ it holds that $\mu \pr{A} = \frac{\nu\pr{A}}{\E \br*{ \zeta }} = d \nu\pr{A}$ prove \cref{it:occ_measure_probability}.
    Next, observe that through approximation by simple functions it holds for all measurable $f \colon [0, \infty) \times U \to \br{-\infty, \infty}$ where $f$ is $\mu$-integrable or satisfies $f \ge 0$ that 
    \begin{equation}
        \int_{[0, \infty) \times U} f\pr{p} \mu \pr{\dx p} = d \int_{[0, \infty) \times U} f\pr{p} \nu \pr{\dx p} = d \E \br*{\int_0^{\zeta} f\pr{s, B_s} \dx s}.
    \end{equation}
    This proves \cref{it:occ_measure_integral}.
    The proof of \cref{lem:occ_measure} is thus complete.
\end{proof}

In the following lemma, a representation of functions of the form \eqref{eq:expectation} via the Walk-on-Spheres process is provided. In the sense of \cref{rem:connection_pde_expectation}, this establishes a Walk-on-Spheres based representation of solutions of the PDE \eqref{eq:elliptic_pde} (see, e.g., \cite[Lemma~3.4]{grohs2022deepelliptic} and \cite[Corollary~2.10]{beznea2022monte} for similar results).

\begin{lemma} \label{lem:WoS_representation_solution}
    Assume $\delta = 0$,
    let $h_0 \colon [0, \infty) \times \partial D \to \R$, $h_1 \colon [0, \infty) \times \ol{D} \to \R$ be measurable and bounded,
    let $w \colon D \to \R$ satisfy for all $x \in D$ that
    \begin{equation}
        w \pr{x} = \E \br*{ h_0 \pr{ \tau^{x}, x + B_{\tau^{x}} } } + \E \br*{ \int_{0}^{\tau^{x}} h_1 \pr{s, x + B_s} \dx s },
    \end{equation}
    let $\mu$  be as in \cref{lem:occ_measure},
    let $\pr{V_k, Y_k} \sim \mu$, $k \in \N$, be independent,
    assume that $\pr{\pr{V_k, Y_k}}_{k \in \N}$ and $B$ are independent.
    Then it holds for all $x \in D$ that
    \begin{multline} \label{eq:WoS_representation_solution}
            w\pr{x} = \E \bigg[ h_0 \pr{\tau^x, x + B_{\tau^x}} \\
            + \frac{1}{d}  \SmallSum{k = 1}{\infty} \mdist\pr{x + B_{\tau_{k - 1}^x}}^2 h_1 \pr*{ \tau^x_{k - 1} +  \mdist\pr{x + B_{\tau_{k - 1}^x}}^2 V_k, x +  B_{\tau_{k - 1}^x} + \mdist\pr{x + B_{\tau_{k - 1}^x}} Y_k }  \bigg]. 
    \end{multline}
\end{lemma}

\begin{proof}
    Throughout this proof
    let $T_c \colon \Omega \to \br{0,\infty}$, $c \in \pr{0, \infty}$, satisfy for every $c \in \pr{0, \infty}$ that $T_c = \inf \cu{ t \in [0,\infty) \colon \norm{B_t} = c}$.
    First, note that the assumption that $\delta = 0$,
    \cite[Remark~2.9]{beznea2022monte},
    the assumption that $h_{1}$ is bounded,
    and the dominated convergence ensure for all $x \in D$ that
    \begin{equation}
        \begin{split} \label{eq:WoS_representation_solution_limit_stopping_times}
            \E \br*{ \int_0^{\tau^x} h_1\pr{s, x + B_s} \dx s } &= \lim_{n \to \infty} \sum_{k = 1}^n \E \br*{ \int_{\tau_{k - 1}^x}^{\tau_k^x} h_1 \pr{s, x + B_s} \dx s }.
        \end{split}
    \end{equation}
    The strong Markov property (cf.,\ \cite[Section~2.6]{karatzas1991brownian}) 
    and the scaling property imply for all $x \in D$, $k \in \N$ that 
    \begin{equation}
        \begin{split} \label{eq:WoS_representation_solution_conditioning}
            &\E \br*{ \int_{\tau_{k - 1}^x}^{\tau_k^x} h_1 \pr{s, x + B_s} \dx s }\\
            &= \E \br*{ \int_0^{\tau_{k}^x - \tau_{k - 1}^x} \hspace{-0cm} h_1 \pr{\tau_{k - 1}^x + s, x + B_{\tau_{k - 1}^x + s}} \dx s  } \\
            &= \E \br*{ \E \br*{ \int_0^{\tau_{k}^x - \tau_{k - 1}^x} h_1 \pr{ \tau_{k - 1}^x + s, x + B_{\tau^x_{k - 1} + s} } \dx s \mid \cF_{\tau_{k - 1}^x} } } \\
            &= \E \br*{ \left.\E\br*{ \int_0^{T_{ \mdist \pr{x + y_0}}} h_1 \pr{ t_0 + s, x + y_0 + B_s }  \dx s  } \right|_{\pr{t_0, y_0} = \pr{\tau_{k - 1}^x, B_{\tau_{k - 1}^x}}} } \\
            &= \E \br*{ \left. \E \br*{ \mdist\pr{x + y_0}^2 \int_0^{T_1} h_1\pr{t_0 + \mdist\pr{x + y_0}^2s, x + y_0 + \mdist\pr{x + y_0} B_{s} } \dx s }  \right|_{\pr{t_0, y_0} = \pr{\tau_{k - 1}^x, B_{\tau_{k - 1}^x}}} }.
        \end{split}
    \end{equation}
    Throughout the remainder of this proof let $U = \cu{ x \in \R^d \colon \norm{x} \le 1}$.
    Furthermore, note that \cref{lem:occ_measure}
    and the assumption that $\pr{V_1, Y_1} \sim \mu$ demonstrate
    for all measurable $\varphi \colon [0, \infty) \times U \to \br{-\infty, \infty}$ where $\varphi$ is $\mu$-integrable or satisfies $\varphi \ge 0$ that  that
    \begin{equation}
        \E \br*{ \int_0^{T_1} \varphi\pr{s, B_s} \dx s} = \frac{1}{d} \int_{[0, \infty) \times U} \varphi \pr{v,y} \mu \pr{\dx \pr{v,y}} = \frac{1}{d} \E \br*{\varphi\pr{V,Y}}.
    \end{equation}
    This, 
    the assumption that $\pr{V_k, Y_k} \sim \mu$, $k \in \N$, are independent,
    the assumption that $\pr{ \pr{V_k, Y_k}}_{k \in \N}$ and $B$ are independent,
    \eqref{eq:WoS_representation_solution_limit_stopping_times},
    and \eqref{eq:WoS_representation_solution_conditioning} establish for all $x \in D$ that  
    \begin{equation}
        \begin{split} \label{eq:WoS_representation_solution_source_part}
            &\E \br*{ \int_0^{\tau^x} h_1\pr{s, x + B_s} \dx s }\\
            &= \frac{1}{d} \E \br*{\sum_{k = 1}^\infty \mdist\pr{x + B_{\tau_{k - 1}^x}}^2 h_1 \pr*{ \tau^x_{k - 1} +  \mdist\pr{x + B_{\tau_{k - 1}^x}}^2 V_k, x +  B_{\tau_{k - 1}^x} + \mdist\pr{x + B_{\tau_{k - 1}^x}} Y_k } }.    
        \end{split}
    \end{equation}
    This proves \eqref{eq:WoS_representation_solution}.
    The proof of \cref{lem:WoS_representation_solution} is thus complete.
\end{proof}

\section{Error analysis for the modified Walk-on-Spheres algorithm} \label{sec:error_mc}

In this section, we establish an error analysis for the estimator \eqref{eq:estimator_intro} as an approximation of expectations of type \eqref{eq:expectation} in the $L^{\infty}$-sense. To this end, we establish in \cref{subsec:error_bias} bias estimates for the Monte Carlo estimator \eqref{eq:estimator_intro}.
In \cref{subsec:error_L_infty} we first present some preparatory regularity properties of the Walk-on-Spheres process as well as the Monte Carlo estimator \eqref{eq:estimator_intro} and conclude with an overall error estimation in the $L^\infty$-sense. The main results are \cref{thm:exp_bound} and \cref{cor:l_infty}.
Throughout this section assume $D \subseteq \R^d$ to be open, bounded, and convex,
let $r \colon \R^d \to [0, \infty)$ satisfy for all $x \in \R^d$ that $r\pr{x} = \inf \cu{ \norm{x - y} \colon y \in \partial D }$,
let $\beta \in (0,1]$, $\delta \in \pr{0,\infty}$,
let $\mdist \colon D \to [0, \infty)$ be a $\pr{\beta, \delta}$-distance on $D$,
let $\tau_{k}^{x}$, $\bar{X}_{k}^{x, j}$, $\bar{\tau}_{k}^{x, j}$, $\zeta$, $Z_{l}^{j}$, $\xi_{l}^{j}$, $\rho_{l}^{j}$, $x \in \ol{D}$, $j,l \in \N$, $k \in \N_0$, be as in \cref{sec:modified_wos},
let $h, h_0, h_1 \colon [0, \infty) \times \ol{D} \to \R$ be measurable and bounded,
let $\mu \colon \cB\pr{[0, \infty) \times \cu{ x \in \R^d \colon \norm{x} \le 1 }} \to \br{0,1}$ be as in \cref{lem:occ_measure}.
For $n \in \N$, $S \subseteq \R^n$, and $F \colon S \to \R$ we denote with $\norm{F}_{\infty}$ the supremum norm of $F$ on its domain $S$, i.e., $\norm{F}_{\infty} = \sup_{s \in S} \abs{F\pr{s}}$.

\subsection{Bias estimates} \label{subsec:error_bias}

\begin{lemma}\label{lem:error_f=0_bias}
    Let $n, M \in \N$, $L_1, L_2 \in \pr{0,\infty}$,
    assume for all $t,s \in [0, \infty)$, $x,y \in \ol{D}$ that
    \begin{equation} \label{eq:error_f=0_bias_h_lip}
        \abs{h \pr{t, x} - h \pr{s, y}} \le L_1 \abs{t - s} + L_2 \norm{x - y},
    \end{equation}
    let $w \colon D \to \R$ satisfy for all $x \in D$ that $w\pr{x} = \E \br*{h \pr{\tau^x, x + B_{\tau^x}}}$,
    and let $\bar{w} = \bar{w}_{\mdist, n, M} \colon D \times \Omega \to \R$ satisfy for all $x \in D$ that
    \begin{equation}
        \bar{w} \pr{x} = \frac{1}{n} \sum_{i = 1}^n h \pr*{ \bar{\tau}^{x, i}_M, \bar{X}^{x, i}_M }.
    \end{equation}
    Then it holds for all $x \in D$ that
    \begin{equation}
        \begin{split} \label{eq:error_f=0_bias}
            \abs*{ w\pr{x} - \E \br*{\bar{w}\pr{x}} } &\le L_1 \diam \pr{D} \delta + L_2 d^{\nicefrac{1}{2}} \diam \pr{D}^{\nicefrac{1}{2}} \delta^{\nicefrac{1}{2}} + 4 \norm{h}_\infty \pr*{1 + \tfrac{\beta^2 \delta^2}{2\diam \pr{D}^2}}^{-M}\!\!\!\!\!.
        \end{split}
    \end{equation}
\end{lemma}

\begin{proof}
    First, observe that \cref{lem:distribution_WoS_BM} shows for all $x \in D$ that
    \begin{equation}
        \begin{split} \label{eq:error_f=0_bias_cases}
            \abs*{ w\pr{x} - \E \br*{\bar{w}\pr{x}} } &= \abs*{\E \br*{ h \pr{\tau^x, x + B_{\tau^x}}} -  \tfrac{1}{n}  \SmallSum{i = 1}{n} \E \br*{ h \pr{\tau_M^{x, i}, \bar{X}^{x, i}_M} } } \\
            &= \abs*{\E \br*{ h \pr{\tau^x, x + B_{\tau^x}} - h \pr{\tau_M^{x}, x + B_{\tau^x_M}  } } } \\
            &\le \abs*{ \E \br*{ \pr*{h \pr{\tau^x, x + B_{\tau^x}} - h \pr{\tau_M^{x}, x + B_{\tau^x_M}  }} \ind_{\cu{M \ge k^x_\delta}} }  } \\
            &\quad + \abs*{ \E \br*{ \pr*{h \pr{\tau^x, x + B_{\tau^x}} - h \pr{\tau_M^{x}, x + B_{\tau^x_M}  }} \ind_{\cu{M < k^x_\delta}} }  }.
        \end{split}
    \end{equation}
    Next, note that combining
    the triangle inequality,
    the fact that $\cu{M < k_\delta^x} \subseteq \cu{M \le k_\delta^x}$,
    and the assumption that $h$ is bounded with \cref{lem:decay} yields for all $x \in D$ that
    \begin{equation}
        \begin{split} \label{eq:error_f=0_bias_M<k}
            &\abs*{ \E \br*{ \pr*{h \pr{\tau^x, x + B_{\tau^x}} - h \pr{\tau_M^{x}, x + B_{\tau^x_M}  }} \ind_{\cu{M < k^x_\delta}} }  }\\
            &\le 2 \norm{h}_{\infty} \E \br{\ind_{\cu{M < k^x_\delta}}} \le 4 \norm{h}_{\infty} \pr*{1 + \tfrac{\beta^2 \delta^2}{2\diam \pr{D}^2}}^{-M}.
        \end{split}
    \end{equation}
    Additionally, note that the assumption that for all $t,s \in [0, \infty)$, $x,y \in D$ it holds that $\abs{ h \pr{t, x} - h \pr{s, y}} \le L_1 \abs{t - s} + L_2 \norm{x - y}$,
    the fact that $\tau^x \ge \tau_M^x$,
    Jensen's inequality,
    the strong Markov property,
    the fact that $\cu{M\ge k_\delta^x} \subseteq \cu{\tau^x_M \ge \tau^x_{k_\delta^x}}$
    and \cref{lem:expected_exit_time_UB},
    imply for all $x \in D$ that
    \begin{equation}
        \begin{split}\label{eq:error_f=0_bias_M>=k}
            &\abs*{ \E \br*{ \pr*{h \pr{\tau^x, x + B_{\tau^x}} - h \pr{\tau_M^{x}, x + B_{\tau^x_M}  }} \ind_{\cu{M \ge k^x_\delta}} }  }\\
            &\le L_1 \E \br*{ \abs{\tau^x - \tau_M^x} \ind_{\cu{M \ge k^x_\delta}}  } + L_2 \E \br*{ \norm{B_{\tau^x} - B_{\tau^x_M} } \ind_{\cu{M \ge k^x_\delta}} } \\
            &\le L_1 \E \br*{ \abs{\tau^x - \tau_M^x} \ind_{\cu{M \ge k^x_\delta}}  } + L_2 \E \br*{ \norm{B_{\tau^x} - B_{\tau^x_M} } \ind_{\cu{\tau^x_M \ge \tau^x_{k^x_\delta}}} } \\
            &= L_1 \E \br*{ \abs{\tau^x - \tau_M^x} \ind_{\cu{M \ge k^x_\delta}}  } + L_2 \E \br*{ \E \br*{\norm{B_{\tau^x} - B_{\tau^x_M} } \ind_{\cu{\tau^x_M \ge \tau^x_{k^x_\delta}}} \mid \mathcal{F}_{\tau^x_M} }} \\
            &= L_1 \E \br*{ \abs{\tau^x - \tau_M^x} \ind_{\cu{M \ge k^x_\delta}}  } + L_2 \E \br*{ \E \br*{\norm{B_{\tau^x} - B_{\tau^x_M} }  \mid \mathcal{F}_{\tau^x_M} } \ind_{\cu{\tau^x_M \ge \tau^x_{k^x_\delta}}} } \\
            &\le L_1 \E \br*{ \abs{\tau^x - \tau_M^x} \ind_{\cu{M \ge k^x_\delta}}  } + L_2 d^{\nicefrac{1}{2}} \pr*{\E \br*{ \E \br*{\abs{\tau^x - \tau_M^x} \mid \mathcal{F}_{\tau^x_M} } \ind_{\cu{\tau^x_M \ge \tau^x_{k^x_\delta}}} }}^{\nicefrac{1}{2}} \\
            &\le L_1 \E \br[\big]{ \abs{\tau^x - \tau^x_{k_\delta^x}} } + L_2 d^{\nicefrac{1}{2}} \pr*{ \E \br[\big]{\abs{\tau^x - \tau^x_{k_\delta^x}}} }^{\!\!\nicefrac{1}{2}} \\
            &\le L_1 \E \br*{ \E \br[\Big]{ \tau^{x + B_{ \tau^{x}_{k^{x}_{\delta}} }}  \mid \mathcal{F}_{\tau^{x}_{k^{x}_{\delta}}}  } } + L_2 d^{\nicefrac{1}{2}} \pr*{\E \br*{ \E \br[\Big]{ \tau^{x + B_{ \tau^{x}_{k^{x}_{\delta}} }}  \mid \mathcal{F}_{\tau^{x}_{k^{x}_{\delta}}} } }}^{\!\!\nicefrac{1}{2}} \\
            &\le  L_1 \diam \pr{D} \delta + L_2 d^{\nicefrac{1}{2}} \diam \pr{D}^{\nicefrac{1}{2}} \delta^{\nicefrac{1}{2}}.
        \end{split}
    \end{equation}
    Combining this with \eqref{eq:error_f=0_bias_M<k} proves \eqref{eq:error_f=0_bias}.
\end{proof}

\begin{lemma} \label{lem:error_g=0}
    Let $n, M \in \N$,
    let $w \colon D \to \R$ satisfy for all $x \in D$ that
    \begin{equation} 
        w \pr{x} = \E \br*{ \int_0^{\tau^x} h\pr{s, x + B_s} \dx s },
    \end{equation}
    let $\pr{V_k^i, Y_k^i} \sim \mu$, $i, k \in \N$, be i.i.d.,
    assume $\pr{V_k^i, Y_k^i}$, $\rho_k^{i}$, $Z_k^{i}$, $i, k \in \N$, and $B$ are independent,
    and let $\bar{w} = \bar{w}_{\mdist, n, M} \colon D \times \Omega \to \R$ satisfy for all $x \in D$ that
    \begin{equation} 
        \bar{w}\pr{x} = \frac{1}{n} \sum_{i = 1}^n \frac{1}{d} \sum_{k = 1}^M \mdist \pr{\bar{X}^{x, i}_{k - 1}}^2 h \pr*{ \bar{\tau}^{x, i}_{k - 1} + \mdist \pr{\bar{X}^{x, i}_{k - 1}}^2 V_k^i, \bar{X}^{x, i}_{k - 1} + \mdist \pr{\bar{X}^{x, i}_{k - 1}} Y_k^i }.
    \end{equation}
    Then
    it holds for all $x \in D$ that
    \begin{equation} \label{eq:error_g=0_bias}
        \begin{split}
            \hspace{-0.7cm} \abs*{ w\pr{x} - \E \br*{\bar{w}\pr{x}} } &\le \norm{h}_{\infty} \diam \pr{D} \pr*{\tfrac{2 \diam\pr{D}}{d} \pr*{1 + \tfrac{\beta^2 \delta^2}{2 \diam \pr{D}^2}}^{-M} \!\! + \delta  }.
        \end{split}
    \end{equation}
\end{lemma}

\begin{proof}
    First, observe that \cref{lem:distribution_WoS_BM,lem:occ_measure},
    the techniques presented in the proof of \cref{lem:WoS_representation_solution},
    and the assumption that $\pr{V_k^1, Y_k^1}$, $k \in \N$, are i.i.d.\ and independent from $B$ 
    establish that for all $x \in D$ it holds that
    \begin{equation}
        \begin{split} \label{eq:error_g=0_solution}
            w\pr{x} &= \E \br*{ \int_{\tau^x_M}^{\tau^x} h \pr{s, x + B_s} \dx s } + \E \br*{ \int_0^{\tau_M^x} h\pr{s, x + B_s} \dx s } \\
            &= \E \br*{ \int_{\tau^x_M}^{\tau^x} h \pr{s, x + B_s} \dx s  } \\
            &\quad+ \frac{1}{d} \sum_{k = 1}^M \E \br*{ \mdist \pr{x + B_{\tau^x_{k - 1}}}^2 h \pr*{ \tau^x_{k - 1} + \mdist \pr{x + B_{\tau^x_{k - 1}}}^2 V_k^1, x + B_{\tau^x_{k - 1}} + \mdist \pr{x + B_{\tau^x_{k - 1}}} Y_k^1 } } \\
            &= \E \br*{ \int_{\tau^x_M}^{\tau^x} h \pr{s, x + B_s} \dx s  } \\
            &\quad+ \frac{1}{d} \sum_{k = 1}^M \E \br*{ \mdist \pr{\bar{X}^{x, 1}_{k - 1}}^2 h \pr*{ \bar{\tau}^{x, 1}_{k - 1} + \mdist \pr{\bar{X}^{x, 1}_{k - 1}}^2 V_k^1, \bar{X}^{x, 1}_{k - 1} + \mdist \pr{\bar{X}^{x, 1}_{k - 1}} Y_k^1 } }.
        \end{split}
    \end{equation}
    Next, note that \cref{lem:distribution_WoS_BM}
    and the assumption that $\pr{V_k^{i}, Y_k^{i}}$, $i,k \in \N$, are i.i.d.\ and independent from $\rho_k^{i}$, $Z_k^{i}$, $i, k \in \N$, and $B$
    imply for all $x \in D$ that
    \begin{equation}
        \begin{split} \label{eq:error_g=0_mc_expectation}
            \E \br*{\bar{w} \pr{x}} &= \frac{1}{nd} \sum_{i = 1}^n \sum_{k = 1}^M \E \br*{ \mdist \pr{ \bar{X}^{x, i}_{k - 1} }^2 h \pr*{ \bar{\tau}^{x, i}_{k - 1} + \mdist \pr{\bar{X}^{x, i}_{k - 1}}^2 V_k^i, \bar{X}^{x, i}_{k - 1} + \mdist \pr{\bar{X}^{x, i}_{k - 1}} Y_k^i } } \\
            &= \frac{1}{d} \sum_{k = 1}^M \E \br*{ \mdist \pr{\bar{X}^{x, 1}_{k - 1}}^2 h \pr*{ \bar{\tau}^{x, 1}_{k - 1} + \mdist \pr{\bar{X}^{x, 1}_{k - 1}}^2 V_k^1, \bar{X}^{x, 1}_{k - 1} + \mdist \pr{\bar{X}^{x, 1}_{k - 1}} Y_k^1 } }.
        \end{split}
    \end{equation}
    Combining \eqref{eq:error_g=0_solution},
    \eqref{eq:error_g=0_mc_expectation},
    and the triangle inequality
    with the assumption that $h$ is bounded yields for all $x \in D$ that
    \begin{equation}
        \begin{split} \label{eq:error_g=0_bias_1}
            \abs*{ w\pr{x} - \E \br*{\bar{w}} } &= \abs*{ \E \br*{ \int_{\tau^x_M}^{\tau^x} h \pr{s, x + B_s} \dx s } } \le \norm{h}_{\infty} \E \br*{ \tau^x - \tau_M^x }.
        \end{split}
    \end{equation}
    Next, observe that 
    the strong Markov property,
    \cite[Corollary~2.3]{beznea2022monte},
    and \cref{lem:decay}
    demonstrate for all $x \in D$ that
    \begin{equation}
        \begin{split}\label{eq:error_g=0_bias_M<K}
            \E \br*{\pr{\tau^x - \tau_M^x} \ind_{ \cu{ M < k_\delta^x }}} &= \E \br*{ \E \br*{\tau^{ x + B_{\tau_M^x} } \mid \cF_{\tau_M^x} } \ind_{ \cu{ M < k_\delta^x }} } \\
            &\le \tfrac{\diam\pr{D}^2}{d} \mathbb{P} \br*{ k_\delta^x \ge M } \le \tfrac{2\diam \pr{D}^2 }{d} \pr*{1 + \tfrac{\beta^2 \delta^2}{2 \diam \pr{D}^2}}^{- M}.
        \end{split}
    \end{equation}
    Moreover, note that 
    the strong Markov property,
    the fact that $\tau^x \ge \tau^x_M$,
    and \cref{lem:expected_exit_time_UB}
    imply for all $x \in D$ that
    \begin{equation}
        \begin{split}
            \E \br*{\pr{\tau^x - \tau^x_M} \ind_{\cu{M \ge k_\delta^x}}} &\le \E \br*{\tau^x - \tau^x_{k_\delta^x}} = \E \br*{ \E \br[\Big]{ \tau^{x + B_{\tau^x_{k^x_\delta}}} \mid \cF_{\tau^x_{k_\delta^x}} }  } \\
            &\le \E \br*{ r \pr*{x + B_{\tau^x_{k^x_\delta}}} \diam \pr{D} } \le \delta \diam \pr{D}.
        \end{split}
    \end{equation}
    Combining this and \eqref{eq:error_g=0_bias_M<K} with \eqref{eq:error_g=0_bias_1} proves \eqref{eq:error_g=0_bias}.
    The proof of \cref{lem:error_g=0} is thus complete.
\end{proof}

\subsection{Global $L^\infty$-error analysis}
\label{subsec:error_L_infty}

\begin{lemma} \label{lem:hoeffding_MC}
    Let $n, M \in \N$,
    let $\pr{V_{k}^{i}, Y_{k}^{i}} \sim \mu$, $i, k \in \N$, be i.i.d.
    assume $\pr{V_{k}^{i}, Y_{k}^{i}}$, $\rho_{k}^{i}$, $Z_{k}^{i}$, $i, k \in \N$, are independent,
    and let $\bar{w} = \bar{w}_{\mdist, n, M} \colon D \times \Omega \to \R$ satisfy for all $x \in D$ that
    \begin{equation}
        \begin{split}
            \bar{w}\pr{x} &= \frac{1}{n} \SmallSum{i = 1}{n} \bigg[ h_0\pr{\bar{\tau}^{x, i}_M, \bar{X}^{x, i}_M} + \frac{1}{d} \SmallSum{k = 1}{M} \mdist\pr{ \bar{X}^{x, i}_{k - 1} }^2 h_1 \pr*{ \bar{\tau}^{x, i}_{k - 1} + \mdist\pr{ \bar{X}^{x, i}_{k - 1} }^2 V_k^{i}, \bar{X}^{x, i}_{k - 1} + \mdist\pr{ \bar{X}^{x, i}_{k - 1} } Y_k^{i} } \bigg].
        \end{split}
    \end{equation}
    Then it holds for all $x \in D$, $t \ge 0$ that 
    \begin{equation}\label{eq:hoeffding_MC}
        \mathbb{P} \br*{ \abs*{\bar{w}\pr{x} - \E\br{ \bar{w} \pr{x}} } \ge t } \le 2 \exp\pr*{- \frac{nt^2}{ 2 \pr[\big]{\norm{h_0}_{\infty} + \tfrac{M \diam \pr{D}^2 \norm{h_1}_{\infty}  }{d}}^2 }}.
    \end{equation}
\end{lemma}

\begin{proof}
    Let $\bar{w}_i = \bar{w}_{i, \delta, n, M} \colon D \times \Omega \to \R$, $i \in \N \cap \br{1,n}$, satisfy for all $i \in \N \cap \br{1,n}$, $x \in D$ that
    \begin{equation}
        \bar{w}_i\pr{x} =  h_0\pr{\bar{\tau}^{x, i}_M, \bar{X}^{x, i}_M} + \frac{1}{d} \sum_{k = 1}^M \mdist\pr{ \bar{X}^{x, i}_{k - 1} }^2 h_1 \pr*{ \bar{\tau}^{x, i}_{k - 1} + \mdist\pr{ \bar{X}^{x, i}_{k - 1} }^2 V_k^{i}, \bar{X}^{x, i}_{k - 1} + \mdist\pr{ \bar{X}^{x, i}_{k - 1} } Y_k^{i} }.
    \end{equation}
    Observe that for all $x \in D$ it holds that $\pr{\bar{w}_i}_{i \in \N \cap \br{1,n}}$ are independent and $\bar{w}\pr{x} = \tfrac{1}{n} \sum_{i = 1}^{n} \bar{w}_i\pr{x}$.
    Moreover, note that the assumption that for all $x \in D$ it holds that $\mdist\pr{x} \le \diam \pr{D}$
    and the assumption that $h_i$, $i \in \cu{1,2}$, is bounded
    ensure for all $i \in \N \cap \br{1,n}$, $x \in D$ it holds that
    \begin{equation}
        \abs{\bar{w}_i\pr{x}} \le \norm{h_0}_{\infty} + \frac{M \diam \pr{D}^2 \norm{h_1}_{\infty}}{d}.
    \end{equation}
    This and Hoeffding's inequality show for all $x \in D$, $t \ge 0$ that
    \begin{equation}
        \begin{split}
            \mathbb{P} \br*{ \abs*{\bar{w}\pr{x} - \E\br{ \bar{w} \pr{x}} } \ge t } &= \mathbb{P} \br*{ \abs*{ \SmallSum{i = 1}{n} \bar{w}_i \pr{x} - \E \br*{ \SmallSum{i = 1}{n} \bar{w}_i\pr{x}}} \ge nt } \\
            &\le 2 \exp\pr*{- \frac{nt^2}{ 2 \pr[\big]{\norm{h_0}_{\infty} + \tfrac{M \diam \pr{D}^2 \norm{h_1}_{\infty}  }{d}}^2 }}.
        \end{split}
    \end{equation}
    This proves \eqref{eq:hoeffding_MC}.
    The proof of \cref{lem:hoeffding_MC} is thus complete.
\end{proof}

\begin{lemma}
    \label{lem:lipschitz_wos_processes}
    Let $L \in [0, \infty)$ and 
    assume for all $x,y \in D$ that $\abs{ \mdist \pr{x} - \mdist\pr{y}} \le L \norm{x - y}$.
    Then
    \begin{enumerate}[label=(\roman*)]
        \item \label{it:lipschitz_wos_processes_space} it holds for all $k \in \N_0$, $i \in \N$, $x,y \in D$ that $\norm{\bar{X}^{x, i}_{k} - \bar{X}^{y, i}_{k}} \le \pr{1 + L}^k \norm{x - y}$ and
        \item \label{it:lipschitz_wos_processes_time} it holds for all $k \in \N_0$, $i \in \N$, $x,y \in D$ that
        \begin{equation} \label{eq:lipschitz_wos_processes_time}
            \abs{\bar{\tau}^{x, i}_k - \bar{\tau}^{y, i}_k} \le 2 \diam \pr{D} L \br*{\SmallSum{j = 1}{k} \pr{1 + L}^{j - 1} \rho_j^{i}} \norm{x - y}.
        \end{equation}
    \end{enumerate}
\end{lemma}

\begin{proof}
    First, we prove \cref{it:lipschitz_wos_processes_space} via induction on $k \in \N_0$.
    Note that for all $i \in \N$, $x, y \in D$ it holds that $\norm{X_0^{x, i} - X_0^{y, i}} = \norm{x - y}$. This proves \cref{it:lipschitz_wos_processes_space} for the base case $k = 0$.
    Let $k \in \N$ and assume \cref{it:lipschitz_wos_processes_space} holds true for all $l \in  \N_0 \cap [0, k -1]$.
    This, the assumption that for all $x, y \in D$ it holds that $\abs{\mdist\pr{x} - \mdist \pr{y}} \le L \norm{x - y}$, and the triangle inequality show for all $i \in \N$, $x, y \in D$ that 
    \begin{equation}
        \begin{split}
            \norm{\bar{X}^{x, i}_k - \bar{X}^{y, i}_k} & \le \norm{\bar{X}^{x, i}_{k - 1} - \bar{X}^{y, i}_{k - 1}} + \norm{\xi_k^{i}} \abs{ \mdist\pr{\bar{X}^{x, i}_{k - 1}} - \mdist\pr{\bar{X}^{y, i}_{k - 1}} } \\
            &\le \pr{1 + L} \norm{\bar{X}^{x, i}_{k - 1} - \bar{X}^{y, i}_{k - 1}} \le \pr{1 + L}^k \norm{x - y}.
        \end{split}
    \end{equation}
    This proves that \cref{it:lipschitz_wos_processes_space} holds true for all $k \in \N_0$.
    Next, we prove \cref{it:lipschitz_wos_processes_time} via induction on $k \in \N_0$. 
    Observe that for all $i \in \N$, $x, y \in D$ it holds that $\abs{\bar{\tau}^{x, i}_0 - \bar{\tau}^{y, i}_0} = 0$. This proves \cref{it:lipschitz_wos_processes_time} for the base case $k = 0$.
    Let $k \in \N$ and assume \cref{it:lipschitz_wos_processes_time} holds true for all $l \in \N_0 \cap \br{0, k - 1}$.
    This, 
    the assumption that for all $x, y \in D$ it holds that $\abs{ \mdist \pr{x} - \mdist\pr{y}} \le L \norm{x - y}$,
    the fact that for all $x \in D$ it holds that $\mdist \pr{x} \in \br{0, \diam \pr{D}}$,
    and \cref{it:lipschitz_wos_processes_space}
    demonstrate for all $i \in \N$, $x, y \in D$ that
    \begin{align}
        \abs{\bar{\tau }^{x, i}_k - \bar{\tau}^{y, i}_k} &\le \abs{\bar{\tau}^{x, i}_{k - 1} - \bar{\tau}^{y, i}_{k - 1}} + \abs{\rho_{k}^{i}} \abs{\mdist \pr{ \bar{X}^{x, i}_{k - 1} }^2 - \mdist \pr{ \bar{X}^{y, i}_{k - 1} }^2} \nonumber \\
        &\le 2 \diam \pr{D} L \br*{\SmallSum{j = 1}{k - 1} \pr{1 + L}^{j - 1} \rho_{j}^{i} } \norm{x - y}  + 2 \diam \pr{D} L \rho_k^{i} \norm{ \bar{X}^{x, i}_{k - 1} - \bar{X}^{x, i}_{k - 1}  } \nonumber \\
        &\le 2 \diam \pr{D} L \br*{ \pr{1 + L}^{k - 1} \rho_k^{i} + \SmallSum{j = 1}{k - 1} \pr{1 + L}^{j - 1} \rho_{j}^{i} } \norm{x - y}.
    \end{align}
    This proves that \cref{it:lipschitz_wos_processes_time} holds true for all $k \in \N_0$.
    The proof of \cref{lem:lipschitz_wos_processes} is thus complete.
\end{proof}

\begin{lemma} \label{lem:occ_measure_expectation_time_component}
    Let $(V, Y) \sim \mu$.
    Then it holds that $\E \br{V} = \frac{d + 4}{2d \pr{d + 2}}$.
\end{lemma}

\begin{mproof}{\cref{lem:occ_measure_expectation_time_component}}
    First, note that $\E \br{ \zeta} = \frac{1}{d}$ (cf.\ \cite[e.g., Proposition~2.2.21~or~3.1.8]{port1978brownian}).
    Next, observe that Itô's formula yields for all $t \in [0, \infty)$ that
    \begin{equation}
        \norm{B_t}^4 - 2\pr{d + 2} \int_0^t \norm{B_s}^2 \dx s = 4 \int_0^t \langle B_s, \dx B_s \rangle.
    \end{equation}
    Let $M = (M_t)_{t \in [0, \infty)}$ satisfy for all $t \in [0, \infty)$ that $M_t = 4 \int_0^{t \wedge \zeta} \langle B_s, \dx B_s \rangle$.
    This and the fact that for all $t \in \br{0, \zeta}$ it holds that $\norm{B_s} \le 1$ implies that the local martingale $M$ is indeed a square integrable martingale.
    Thus it holds for all $t \in [0,\infty)$ that
    \begin{equation}
        \E \br{\norm{B_{t \wedge \zeta}}^4} - 2\pr{2 + d} \E \br*{ \int_0^{t \wedge \zeta} \norm{B_s}^2 \dx s } = 0.
    \end{equation}
    This, the dominated convergence theorem, 
    and the monotone convergence theorem imply that
    \begin{equation} \label{eq:exit_BM_unit_sphere_second_moment_ito4}
        1 = 2\pr{2 + d} \E \br*{ \int_0^\zeta \norm{B_s}^2 \dx s }.
    \end{equation}
    Similarly, Itô's product rule establishes for all $t \in [0, \infty)$ that
    \begin{equation}
        t \norm{B_t}^2 - \frac{dt^2}{2} - \int_0^{t} \norm{B_s}^2 \dx s = 2 \int_{0}^t s \langle B_s, \dx B_s \rangle.
    \end{equation}
    Again this defines a local martingale, which stopped at $\zeta$ is a square integrable martingale.
    This, the dominated convergence theorem, and the monotone convergence theorem hence demonstrate that
    \begin{equation}
        \E \br{\zeta} - \frac{d}{2} \E \br{\zeta^{2}} - \E \br*{ \int_0^\zeta \norm{B_s}^2 \dx s} = 0.
    \end{equation}
    Combining this, the fact that $\E \br{\zeta} = \frac{1}{d}$, and $\eqref{eq:exit_BM_unit_sphere_second_moment_ito4}$ yields $\E \br{\zeta^{2}} = \frac{d + 4}{d^2 \pr{d + 2}}$.
    This and \cref{it:occ_measure_integral} in \cref{lem:occ_measure} establish that
    \begin{equation}
        \E \br{V} = \int_{ [0, \infty) \times U } v \mu \pr{ \dx \pr{v, y}} = d \E \br*{\int_0^\zeta s \dx s} = \frac{d}{2} \E \br{\zeta^2} = \frac{d + 4}{2 d \pr{d + 2}}.
    \end{equation}
    The proof of \cref{lem:occ_measure_expectation_time_component} is thus complete.
\end{mproof}

\begin{lemma}\label{lem:lipschitz_MC}
    Let $n, M \in \N$,
    let $\pr{V_k^i, Y_k^i} \sim \mu$, $i, k \in \N$, be i.i.d.,
    assume $\pr{V_k^i, Y_k^i}$, $\rho_k^{i}$, $Z_k^{i}$, $i, k \in \N$, are independent,
    let $L, L_j^{i} \in [0, \infty)$, $i \in \cu{0,1}$, $j \in \cu{1,2}$,
    assume for all $s,t \in [0, \infty)$, $x, y \in D$, $i \in \cu{0,1}$ that
    \begin{equation}
        \abs{ h_i \pr{s, x} - h_i\pr{t,y} } \le L_1^{i} \abs{ s- t} + L_2^{i} \norm{x - y} \qquad \text{and} \qquad \abs{\mdist\pr{x} - \mdist \pr{y}} \le L \norm{x - y},
    \end{equation}
    let $\bar{w} = \bar{w}_{\mdist, n, M} \colon D \times \Omega \to \R$ satisfy for all $x \in D$ that
    \begin{equation}
        \begin{split}
            \bar{w}\pr{x} &= \frac{1}{n} \SmallSum{i = 1}{n} \bigg[ h_0\pr{\bar{\tau}^{x, i}_M, \bar{X}^{x, i}_M} + \frac{1}{d} \SmallSum{k = 1}{M} \mdist\pr{ \bar{X}^{x, i}_{k - 1} }^2 h_1 \pr*{ \bar{\tau}^{x, i}_{k - 1} + \mdist\pr{ \bar{X}^{x, i}_{k - 1} }^2 V_k^{i}, \bar{X}^{x, i}_{k - 1} + \mdist\pr{ \bar{X}^{x, i}_{k - 1} } Y_k^{i} } \bigg].
        \end{split}
    \end{equation}
    Then
    \begin{enumerate}[label=(\roman*)]
        \item \label{it:lipschitz_MC_E_w} it holds for all $x, y \in D$ that
        \begin{equation}
            \begin{split} \label{eq:lipschitz_MC_E_w}
                \hspace{-0.3cm}\abs{\E \br{ \bar{w}\pr{x}} - \E \br{\bar{w}\pr{y}}} &\le \norm{ x - y} \pr{1 + L}^M M^2 \bigg( \! L_2^0 + \tfrac{2\diam\pr{D}L_1^0}{d} \\
                &\quad+ \tfrac{1}{d} \pr*{ 2 \diam\pr{D} \norm{h_1}_{\infty} + \diam\pr{D}^2 \pr*{ L_2^1 + \tfrac{2\diam \pr{D} L_1^1}{d} } } \!\bigg)
            \end{split}
        \end{equation}
        \item \label{it:lipschitz_MC_w} it holds for all $x, y \in D$ that
        \begin{equation}
            \begin{split}
                \hspace{-0.3cm}\abs{\bar{w}\pr{x} - \bar{w}\pr{y}} &\le \norm{x - y} \pr{1 + L}^M \bigg( L_2^0 + \tfrac{2L_1^0 \diam \pr{D}}{n} \SmallSum{i = 1}{n} \SmallSum{j = 1}{M} \rho_j^{i} + \tfrac{2M\diam \pr{D} \norm{h_1}_{\infty}}{d} \\
                &\quad+ \tfrac{2L_1^1 \diam \pr{D}^3}{dn} \SmallSum{i = 1}{n} \SmallSum{k = 1}{M} \br[\Big]{V_k^{i} + \SmallSum{j = 1}{k - 1} \rho_j^{i}} + \tfrac{L_2^1 M \diam\pr{D}^2}{d}\bigg).
            \end{split}
        \end{equation}
    \end{enumerate}
\end{lemma}

\begin{proof}
    First, note that \cref{lem:distribution_WoS_BM},
    the assumption that $\pr{V_{k}^{i}, Y_{k}^{i}}$, $i, k \in \N$ are identically distributed,
    and the assumption that $\pr{V_{k}^{i}, Y_{k}^{i}}$, $\rho_{k}^{i}, Z_{k}^{i}$, $i,k \in \N$, are independent
    ensure for all $x \in D$ that
    \begin{equation}
        \begin{split}
            \E \br{\bar{w} \pr{x}} &= \E \br*{ h_0 \pr{ \bar{\tau}^{x, 1}_M, \bar{X}^{x, 1}_M } } \\
            &\quad+ \tfrac{1}{d} \SmallSum{k = 1}{M} \E \br*{ \mdist\pr{\bar{X}^{x, 1}_{k - 1}}^2 h_1 \pr*{ \bar{\tau}^{x, 1}_{k - 1} + \mdist\pr{ \bar{X}^{x, 1}_{k - 1} }^2 V_k^{1}, \bar{X}^{x, 1}_{k - 1} + \mdist\pr{ \bar{X}^{x, 1}_{k - 1} } Y_k^{1}  } }.
        \end{split}
    \end{equation}
    This assures for all $x, y \in D$ that
    \begin{equation}
        \begin{split}
            \abs*{ \E \br{\bar{w} \pr{x}} - \E \br{ \bar{w} \pr{y} } } &\le \E \br[\Big]{ \abs*{ h_0 \pr{\bar{\tau}^{x, 1}_{M}, \bar{X}^{x, 1}_M} - h_0 \pr{\bar{\tau}^{y, 1}_{M}, \bar{X}^{y, 1}_M} } } \\
            &\quad + \tfrac{1}{d} \SmallSum{k = 1}{M} \E \Big[ \big| \mdist\pr{\bar{X}^{x, 1}_{k - 1}}^2 h_1 \pr*{ \bar{\tau}^{x, 1}_{k - 1} + \mdist\pr{ \bar{X}^{x, 1}_{k - 1} }^2 V_k^{1}, \bar{X}^{x, 1}_{k - 1} + \mdist\pr{ \bar{X}^{x, 1}_{k - 1} } Y_k^{1}  } \\
            &\quad-  \mdist\pr{\bar{X}^{y, 1}_{k - 1}}^2 h_1 \pr*{ \bar{\tau}^{y, 1}_{k - 1} + \mdist\pr{ \bar{X}^{y, 1}_{k - 1} }^2 V_k^{1}, \bar{X}^{y, 1}_{k - 1} + \mdist\pr{ \bar{X}^{y, 1}_{k - 1} } Y_k^{1}  } \big|  \Big].
        \end{split}
    \end{equation}
    Next, observe that the assumption that for all $i \in \cu{0,1}$, $s,t \in [0, \infty)$, $x, y \in D$ it holds that $\abs{ h_i \pr{s, x} - h_i\pr{t,y} } \le L_1^{i} \abs{ s- t} + L_2^{i} \norm{x - y}$
    and the fact that for all $x \in D$ it holds that $\mdist \pr{x} \le \diam \pr{D}$
    imply for all $k \in \N \cap \br{1,M}$, $x, y \in D$ that 
    \begin{equation}
        \begin{split}
            &\big|  \mdist\pr{\bar{X}^{x, 1}_{k - 1}}^2 h_1 \pr*{ \bar{\tau}^{x, 1}_{k - 1} + \mdist\pr{ \bar{X}^{x, 1}_{k - 1} }^2 V_k^{1}, \bar{X}^{x, 1}_{k - 1} + \mdist\pr{ \bar{X}^{x, 1}_{k - 1} } Y_k^{1}  } \\
            &\quad-  \mdist\pr{\bar{X}^{y, 1}_{k - 1}}^2 h_1 \pr*{ \bar{\tau}^{y, 1}_{k - 1} + \mdist\pr{ \bar{X}^{y, 1}_{k - 1} }^2 V_k^{1}, \bar{X}^{y, 1}_{k - 1} + \mdist\pr{ \bar{X}^{y, 1}_{k - 1} } Y_k^{1}  } \big| \\
            &\le 2 \diam \pr{D} \norm{h_1}_{\infty}  \abs{ \mdist \pr{\bar{X}^{x, 1}_{k - 1}} - \mdist \pr{\bar{X}^{y, 1}_{k - 1}} }\\
            &\quad+ \diam\pr{D}^2 \bigg( L_1^1 \pr*{\abs{ \bar{\tau}^{x,1}_{k - 1} - \bar{\tau}^{y,1}_{k - 1} } + \abs{V_k^1} \abs{ \mdist \pr{ \bar{X}^{x, 1}_{k - 1} }^2 -\mdist \pr{\bar{X}^{y, 1}_{k - 1}}^2   } }\\
            &\quad + L_2^1 \pr*{ \norm{\bar{X}^{x, 1}_{k - 1} - \bar{X}^{y, 1}_{k - 1}} + \norm{Y_k^1} \abs{\mdist \pr{ \bar{X}^{x, 1}_{k - 1} } -\mdist \pr{\bar{X}^{y, 1}_{k - 1}}   } } \bigg).
        \end{split}
    \end{equation}
    This,
    the assumption that for all $x, y \in D$ it holds that $\abs{\mdist \pr{x} - \mdist \pr{y}} \le L \norm{x - y}$,
    the fact that $\E \br{\zeta} = \frac{1}{d}$ (cf.\ \cite[e.g., Proposition~2.2.21 or 3.1.8]{port1978brownian}),
    the assumption that for all $i,j \in \N$ it holds that $\rho_i^j$ and $\zeta$ are identically distributed,
    \cref{it:occ_measure_integral} in \cref{lem:occ_measure},
    \cref{lem:occ_measure_expectation_time_component},
    the fact that for all $k \in \N$ it holds that $\tfrac{k + 4}{2k(k + 2)} \le \tfrac{1}{k}$,
    the fact that for all $j,k \in \N$ it holds that $\norm{Y_k^j} \le 1$,
    and \cref{lem:lipschitz_wos_processes}
    establish for all $x, y \in D$ that
    \begin{equation}
        \begin{split} \label{eq:lipschitz_MC_source_expectation}
            &\tfrac{1}{d} \SmallSum{k = 1}{M} \E \Big[ \big| \mdist\pr{\bar{X}^{x, 1}_{k - 1}}^2 h_1 \pr*{ \bar{\tau}^{x, 1}_{k - 1} + \mdist\pr{ \bar{X}^{x, 1}_{k - 1} }^2 V_k^{1}, \bar{X}^{x, 1}_{k - 1} + \mdist\pr{ \bar{X}^{x, 1}_{k - 1} } Y_k^{1}  } \\
            &\quad-  \mdist\pr{\bar{X}^{y, 1}_{k - 1}}^2 h_1 \pr*{ \bar{\tau}^{y, 1}_{k - 1} + \mdist\pr{ \bar{X}^{y, 1}_{k - 1} }^2 V_k^{1}, \bar{X}^{y, 1}_{k - 1} + \mdist\pr{ \bar{X}^{y, 1}_{k - 1} } Y_k^{1}  } \big|  \Big] \\
            &\le \norm{x - y} \tfrac{1}{d} \SmallSum{k = 1}{M} \Big[ \pr{1 + L}^k \Big( 2 \diam \pr{D} \norm{h_1}_{\infty}  \\
            &\quad+ \diam \pr{D}^2 \Big( L_2^1 + 2L_1^1 \diam \pr{D} \pr[\Big]{ \E \br{V_k^1} + \SmallSum{j = 1}{k-1} \E \br{\rho_j^1} } \Big) \Big) \Big] \\
            &\le \norm{x - y} \pr{1+L}^M \tfrac{M^2}{d} \Big( 2 \diam\pr{D} \norm{h_1}_{\infty} + \diam \pr{D}^2 \pr*{L_2^1 + \tfrac{2L_1^1 \diam\pr{D}}{d}} \Big).
        \end{split}
    \end{equation}
    Similarly, it holds for all $x, y \in D$ that
    \begin{equation}
        \begin{split}
            &\E \br[\Big]{ \abs*{ h_0 \pr{\bar{\tau}^{x, 1}_{M}, \bar{X}^{x, 1}_M} - h_0 \pr{\bar{\tau}^{y, 1}_{M}, \bar{X}^{y, 1}_M} } } \le \norm{x - y} \pr{1 + L}^M \pr*{L_2^0 + \tfrac{2 M L_1^0 \diam\pr{D}}{d}}.
        \end{split}
    \end{equation}
    Combining this and \eqref{eq:lipschitz_MC_source_expectation} yields for all $x, y \in D$ that
    \begin{equation}
        \begin{split}
            \abs{\E \br{ \bar{w}\pr{x}} - \E \br{\bar{w}\pr{y}}} &\le \norm{x - y} \pr{1 + L}^M M^2 \Big(L_2^0 + \tfrac{2L_1^0 \diam \pr{D}}{d}\\
            &\quad+ \tfrac{1}{d} \pr*{ 2 \diam\pr{D} \norm{h_1}_{\infty} + \diam\pr{D}^2 \pr*{ L_2^1 + \tfrac{2L_1^1\diam \pr{D}}{d} } } \Big)
        \end{split}
    \end{equation}
    This proves \cref{it:lipschitz_MC_E_w}. Moreover, note that \cref{it:lipschitz_MC_w} is proven analogously.
    The proof of \cref{lem:lipschitz_MC} is thus complete.
\end{proof}

\begin{lemma}\label{lem:tail_probability_times}
    Let $\pr{V,Y} \sim \mu$.
    Then it holds for all $t \in [0, \infty)$ that
    \begin{equation} \label{eq:tail_probability_times}
        \mathbb{P} \br{\zeta \ge t} \le 2 e^{- \frac{d}{2}t} \qquad \text{and} \qquad \mathbb{P} \br{V \ge t} \le 4 e^{- \frac{d}{2}t}.
    \end{equation}
\end{lemma}

\begin{mproof}{\cref{lem:tail_probability_times}}
    First, observe that \cref{lem:mgf_exit_time_BM} yields that $\E \br{e^{\frac{d}{2}\zeta}} \le 2$.
    Combining this with Markov's inequality demonstrates for all $t \in [0,\infty)$ that
    \begin{equation} \label{eq:tail_probability_times_exit}
        \mathbb{P} \br{\zeta \ge t} \le \E \br{e^{\frac{d}{2}\zeta}} e^{-\frac{d}{2}t} \le 2e^{-\frac{d}{2}t}.
    \end{equation}
    Next, observe that \cref{it:occ_measure_integral} in \cref{lem:occ_measure} and Fubini's theorem establish for all $t \ge 0$ that
    \begin{equation}
        \mathbb{P} \br{V \ge t} = \E \br{ \ind_{[t, \infty)}\pr{V} } = d \E \br*{\int_0^\zeta \ind_{[t, \infty)} \pr{s} \dx s} = d \int_t^\infty \mathbb{P} \br{\zeta \ge s} \dx s. 
    \end{equation}
    This and \eqref{eq:tail_probability_times_exit} show for all $t \in [0, \infty)$ that
    \begin{equation}
        \mathbb{P} \br{V \ge t} \le 2d \int_t^\infty e^{-\frac{d}{2}s} \dx s = 4 e^{-\frac{d}{s}t}.
    \end{equation}
    The proof of \cref{lem:tail_probability_times} is thus complete.
\end{mproof}

\bigskip \noindent
The next lemma is a modification
of \cite[Lemma~4.5]{beznea2022monte}.
We include a proof for the convenience of the reader.

\begin{lemma} \label{lem:expectation_bound_by_tail_bound}
    Let $X\colon \Omega \to [0, \infty)$ be a random variable,
    let $c_1, c_2, c_3, c_4, A \ge 0$,
    assume $c_2, c_4 > 0$,
    assume for all $t \in [0, \infty)$ that $\mathbb{P} \br{ X \ge t} \le 2 \pr[\big]{ e^{ c_1 - c_2 \pr{ \pr{t - A}^+ }^2} + e^{c_3 - c_4 \pr{t - A}^+}}$.
    Then it holds that
    \begin{equation}
        \E \br{X} \le A + \frac{\sqrt{c_1 + \log\pr{2}} + 1}{\sqrt{c_2}} + \frac{c_3 + \log\pr{2} + 1}{c_4}.
    \end{equation}
\end{lemma}

\begin{mproof}{\cref{lem:expectation_bound_by_tail_bound}}
    Throughout this proof let $\lambda \in \R$ satisfy
    \begin{equation}\label{eq:expectation_bound_by_tail_bound_def_lambda}
        \lambda = A + \tfrac{\sqrt{c_1 + \log\pr{2}}}{\sqrt{c_2}} + \tfrac{c_3 + \log\pr{2}}{c_4}.
    \end{equation}
    Observe that $\lambda > A$.
    This and the assumption that for all $t \in [0, \infty)$ it holds that $\mathbb{P} \br*{ X \ge t } \le 2 \pr{ e^{ c_1 - c_2 \pr{ \pr{t - A}^+ }^2 } + e^{c_3 - c_4 \pr{t - A}^+} }$ ensures that
    \begin{equation}
        \begin{split} \label{eq:exppectation_bound_by_tail_bound_tail_integral}
            \E \br{X} &= \int_0^\infty \mathbb{P} \br*{X \ge t} \dx t \le \lambda + \int_{\lambda}^\infty \mathbb{P} \br*{X \ge t} \dx t \\
            &\le \lambda + 2e^{c_1} \int_{\lambda}^\infty e^{-c_2 (t - A)^2} \dx t + 2e^{c_3} \int_{\lambda}^\infty e^{-c_4 (t - A)} \dx t.
        \end{split}
    \end{equation}
    Next, note that the fact that it holds that $\lambda - A \ge \tfrac{c_3 + \log \pr{2}}{c_4}$ implies that
    \begin{equation}
        \begin{split} \label{eq:expectation_bound_by_tail_bound_exponetial_estimate}
            2 e^{c_3} \int_{\lambda}^\infty e^{- c_4 (t - A)} \dx t = \frac{2}{c_4} e^{c_3 -c_4 (\lambda - A)} \le \frac{2}{c_4} e^{-\log\pr{2}} = \frac{1}{c_4}.
        \end{split}
    \end{equation}
    Moreover, observe that the fact that $\lambda - A \ge \tfrac{\sqrt{c_1 + \log\pr{2}}}{\sqrt{c_2}}$ and the fact that $2 \sqrt{\log\pr{2}} \ge 1$ establish that
    \begin{equation}
        \begin{split}
            2e^{c_1} \int_\lambda^\infty e^{-c_2(t - A)^2} \dx t &= 2e^{c_1} \int_{\lambda - A}^\infty e^{- c_2 s^2} \dx s \le \frac{2e^{c_1}}{\lambda - A} \int_{\lambda - A}^\infty s e^{-c_2 s^2} \dx s \\
            &\le \frac{e^{c_1}}{(\lambda - A)c_2} e^{-c_2(\lambda - A)^2} \le \frac{e^{- \log \pr{2}}}{\sqrt{c_2} \sqrt{c_1 + \log(2)}} \le \frac{1}{\sqrt{c_2}}.
        \end{split}
    \end{equation}
    Combining this, \eqref{eq:exppectation_bound_by_tail_bound_tail_integral}, and \eqref{eq:expectation_bound_by_tail_bound_exponetial_estimate} with \eqref{eq:expectation_bound_by_tail_bound_def_lambda} yields
    \begin{equation}
        \begin{split}
            \E \br{X} \le \lambda + \frac{1}{\sqrt{c_2}} + \frac{1}{c_4} = A + \frac{\sqrt{c_1 + \log\pr{2}} + 1}{\sqrt{c_2}} + \frac{c_3 + \log\pr{2} + 1}{c_4}. 
        \end{split}
    \end{equation}
    \bigskip \noindent
    The proof of \cref{lem:expectation_bound_by_tail_bound} is thus complete.
\end{mproof}

The following Theorem demonstrates that the estimator \eqref{eq:estimator_intro} approximates functions \eqref{eq:expectation} in the $L^\infty$-sense. The proof follows similar arguments as those presented in \cite[Proof of Theorem~2.26]{beznea2022monte}.

\begin{theo} \label{thm:exp_bound}
    Let $n, M \in \N$,
    let $\pr{V_k^i, Y_k^i} \sim \mu$, $i, k \in \N$, be i.i.d.,
    assume $\pr{V_k^i, Y_k^i}$, $\rho_k^{i}$, $Z_k^{i}$, $i, k \in \N$, and $B$ are independent,
    let $L, L_j^{i} \in [0, \infty)$, $i \in \cu{0,1}$, $j \in \cu{1,2}$,
    assume for all $s,t \in [0, \infty)$, $x, y \in D$, $i \in \cu{0,1}$ that
    \begin{equation}
        \abs{ h_i \pr{s, x} - h_i\pr{t,y} } \le L_1^{i} \abs{ s- t} + L_2^{i} \norm{x - y} \qquad \text{and} \qquad \abs{\mdist\pr{x} - \mdist \pr{y}} \le L \norm{x - y},
    \end{equation}
    let $w \colon D \to \R$ be continuous and satisfy for all $x \in D$ that
    \begin{equation}
        w\pr{x} = \E \br*{h_0 \pr{\tau^x, x + B_{\tau^x}}} + \E \br*{\int_0^{\tau^x} h_1\pr{s, x + B_s} \dx s},
    \end{equation}
    let $\bar{w} = \bar{w}_{\mdist, n, M} \colon D \times \Omega \to \R$ satisfy for all $x \in D$ that
    \begin{equation}
        \begin{split}
            \bar{w}\pr{x} &= \frac{1}{n} \SmallSum{i = 1}{n} \bigg[ h_0\pr{\bar{\tau}^{x, i}_M, \bar{X}^{x, i}_M} + \frac{1}{d} \SmallSum{k = 1}{M} \mdist\pr{ \bar{X}^{x, i}_{k - 1} }^2 h_1 \pr*{ \bar{\tau}^{x, i}_{k - 1} + \mdist\pr{ \bar{X}^{x, i}_{k - 1} }^2 V_k^{i}, \bar{X}^{x, i}_{k - 1} + \mdist\pr{ \bar{X}^{x, i}_{k - 1} } Y_k^{i} } \bigg].
        \end{split}
    \end{equation}
    Then it holds for all $K \in \N$ that
    \begin{equation}
        \begin{split} \label{eq:exp_bound}
            &\E \br*{\sup_{x \in D} \abs{w\pr{x} - \bar{w}\pr{x}}} \\
            &\le \pr{L_1^0 + \norm{h_1}_{\infty}} \diam \pr{D} \delta + L_2^0 d^{\nicefrac{1}{2}} \diam \pr{D}^{\nicefrac{1}{2}} \delta^{\nicefrac{1}{2}} \\
            &\quad+ \pr*{4 \norm{h_0}_{\infty} + \tfrac{ 2 \diam \pr{D}^2 \norm{h_1}_{\infty}}{ d} }\pr*{1 + \tfrac{\beta^2 \delta^2}{2 \diam \pr{D}^2}}^{-M} \\
            &\quad+ \pr*{L_2^0 + \tfrac{2L_1^0 \diam \pr{D}}{d} + \tfrac{1}{d} \pr*{ 2 \diam \pr{D} \norm{h_1}_{\infty} + \diam \pr{D}^2 \pr*{L_2^1 + \tfrac{2L_1^1 \diam \pr{D}}{d}} }} \tfrac{d^{\nicefrac{1}{2}} \diam \pr{D}}{K} \\
            &\quad+ \pr*{L_2^0 + \tfrac{2M \diam \pr{D} \norm{h_1}_{\infty} + L_2^1 M \diam \pr{D}^2}{d}} \tfrac{d^{\nicefrac{1}{2}} \diam \pr{D} }{K M^2} \\
            &\quad+ \tfrac{ 3\sqrt{2} \pr*{\sqrt{d\pr*{\log\pr{K} + 2 \log\pr{M} + M \log\pr{1 + L}}  + \log\pr{2}} + 1} \pr*{\norm{h_0}_{\infty} + \tfrac{M \diam \pr{D}^2 \norm{h_1}_{\infty}}{d}} }{\sqrt{n}} \\
            &\quad+ \tfrac{12 \diam \pr{D}^2 \pr{ \log\pr{n} + \log\pr{M}+ \log\pr{6}+ 1 } \pr{ L_1^0 d + L_1^1 \diam \pr{D}^2 M }}{d^{\nicefrac{3}{2}} K M}.
        \end{split}
    \end{equation}
\end{theo}

\begin{proof}
We first note that the supremum $\sup_{x \in D} \abs{w\pr{x} - \bar{w}\pr{x}}$ is measurable because of the continuity of $w$ and $\bar{w}$.
    Throughout this proof, let $K \in \N$ be fixed and assume without loss of generality that $D \subseteq \br{0, \diam \pr{D}}^d$. Let $\fD \subseteq \R^d$ satisfy that
    \begin{equation} \label{eq:exp_bound_grid_def}
        \fD = \cu*{\tfrac{i \diam \pr{D}}{KM^2\pr{1 + L}^M} \colon i \in \N_0 \cap \br{0, K M^2 \pr{1 + L}^M}}^d \cap D.
    \end{equation}
    Note that for all $x \in D$ there exist $i_1, \dots, i_d \in \N_0 \cap \br{0, KM^2 \pr{1 +L}^M}$ such that
    \begin{equation}
        x \in \left[ \tfrac{i_1 \diam \pr{D}}{K M^2 \pr{1 + L}^M}, \tfrac{\pr{i_1 + 1} \diam \pr{D}}{K M^2 \pr{1 + L}^M} \right) \times \dots \times \left[ \tfrac{i_d \diam \pr{D}}{K M^2 \pr{1 + L}^M}, \tfrac{\pr{i_d + 1} \diam \pr{D}}{K M^2 \pr{1 + L}^M} \right).
    \end{equation}
    This implies for all $x \in D$ that there exists $x_\fD \in \fD$ such that $\norm{x - x_\fD} \le d^{\nicefrac{1}{2}} \frac{\diam \pr{D}}{K M^2 \pr{1 + L}^M}$.
    This,
    \cref{lem:error_f=0_bias},
    \cref{lem:error_g=0},
    and \cref{lem:lipschitz_MC} demonstrate
    \begin{equation}
        \begin{split} \label{eq:exp_bound_sup_error}
            &\sup_{x \in D} \abs{ w\pr{x} - \bar{w}\pr{x} }\\
            &\le \sup_{x \in D} \abs*{ w \pr{x} - \E \br{ \bar{w}\pr{x}} } + \sup_{x \in D} \abs*{\bar{w} \pr{x} - \bar{w}\pr{x_\fD}} + \sup_{x \in D} \abs*{\E \br{ \bar{w} \pr{x} } - \E \br{\bar{w} \pr{x_\fD}}} \\
            &\quad+ \sup_{y \in \fD} \abs*{\bar{w} \pr{y} - \E \br{\bar{w} \pr{y}}} \\
            &\le \pr{L_1^0 + \norm{h_1}_{\infty}} \diam \pr{D} \delta + L_2^0 d^{\nicefrac{1}{2}} \diam \pr{D}^{\nicefrac{1}{2}} \delta^{\nicefrac{1}{2}} \\
            &\quad+ \pr*{4 \norm{h_0}_{\infty} + \tfrac{ 2 \diam \pr{D}^2 \norm{h_1}_{\infty}}{ d} }\pr*{1 + \tfrac{\beta^2 \delta^2}{2 \diam \pr{D}^2}}^{-M} \\
            &\quad+ \pr*{L_2^0 + \tfrac{2L_1^0 \diam \pr{D}}{d} + \tfrac{1}{d} \pr*{ 2 \diam \pr{D} \norm{h_1}_{\infty} + \diam \pr{D}^2 \pr*{L_2^1 + \tfrac{2L_1^1 \diam \pr{D}}{d}} }} \tfrac{d^{\nicefrac{1}{2}} \diam \pr{D}}{K} \\
            &\quad+ \pr*{L_2^0 + \tfrac{2M \diam \pr{D} \norm{h_1}_{\infty} + L_2^1 M \diam \pr{D}^2}{d}} \tfrac{d^{\nicefrac{1}{2}} \diam \pr{D} }{K M^2} \\
            &\quad+ \pr*{ \tfrac{2 L_1^0 \diam \pr{D}}{n} \sum_{i = 1}^{n} \sum_{j = 1}^M \rho_j^{i} + \tfrac{2L_1^1 \diam \pr{D}^3}{dn} \sum_{i = 1}^n \sum_{k = 1}^M \br*{ V_k^{i} + \sum_{j = 1}^{k - 1} \rho_j^{i} } } \tfrac{d^{\nicefrac{1}{2}} \diam \pr{D}}{K M^2} \\
            &\quad+ \sup_{y \in \fD} \abs*{\bar{w} \pr{y} - \E \br{\bar{w} \pr{y}}} \\
            &= C \pr{h_0, h_1, \delta, d, D, K, M} + Z,
        \end{split}
    \end{equation}
    where $C \pr{h_0, h_1, \delta, d, D, K, M} \in [0, \infty)$ and $Z \colon \Omega \to [0, \infty)$ satisfy
    \begin{equation}
        \begin{split} 
            &C \pr{h_0, h_1, \delta, d, D, K, M} \\
            &= \pr{L_1^0 + \norm{h_1}_{\infty}} \diam \pr{D} \delta + L_2^0 d^{\nicefrac{1}{2}} \diam \pr{D}^{\nicefrac{1}{2}} \delta^{\nicefrac{1}{2}} \\
            &\quad+ \pr*{4 \norm{h_0}_\infty + \tfrac{ 2 \diam \pr{D}^2 \norm{h_1}_{\infty}}{ d} }\pr*{1 + \tfrac{\beta^2 \delta^2}{2 \diam \pr{D}^2}}^{-M} \\
            &\quad+ \pr*{L_2^0 + \tfrac{2L_1^0 \diam \pr{D}}{d} + \tfrac{1}{d} \pr*{ 2 \diam \pr{D} \norm{h_1}_{\infty} + \diam \pr{D}^2 \pr*{L_2^1 + \tfrac{2L_1^1 \diam \pr{D}}{d}} }} \tfrac{d^{\nicefrac{1}{2}} \diam \pr{D}}{K} \\
            &\quad+ \pr*{L_2^0 + \tfrac{2M \diam \pr{D} \norm{h_1}_{\infty} + L_2^1 M \diam \pr{D}^2}{d}} \tfrac{d^{\nicefrac{1}{2}} \diam \pr{D} }{K M^2}
        \end{split}
    \end{equation}
    and 
    \begin{equation}
        \begin{split}
            Z &= \pr*{ \tfrac{2 L_1^0 \diam \pr{D}}{n} \sum_{i = 1}^{n} \sum_{j = 1}^M \rho_j^{i} + \tfrac{2L_1^1 \diam \pr{D}^3}{dn} \sum_{i = 1}^n \sum_{k = 1}^M \br*{ V_k^{i} + \sum_{j = 1}^{k - 1} \rho_j^{i} } } \tfrac{d^{\nicefrac{1}{2}} \diam \pr{D}}{K M^2} \\
            &\quad+ \sup_{y \in \fD} \abs*{\bar{w} \pr{y} - \E \br{\bar{w} \pr{y}}}.
        \end{split}
    \end{equation}
    Moreover, note that the fact that for all $i \in \N$ it holds that $\sum_{k = 1}^M \sum_{j = 1}^{k - 1} \rho_k^{i} \le M \sum_{k = 1}^M \rho_j^{i}$ ensures that
    \begin{equation} \label{eq:exp_bound_Z_bound}
        Z \le \tfrac{2 \diam \pr{D}^2 \pr{L_1^0 d + L_1^1 \diam \pr{D}^2 M} }{d^{\nicefrac{1}{2}} n K M^2 } \SmallSum{i = 1}{n} \SmallSum{k = 1}{M} \rho_k^{i} + \tfrac{2 L_1^1 \diam \pr{D}^4}{d^{\nicefrac{1}{2}} n K M^2} \SmallSum{i = 1}{n} \SmallSum{k = 1}{M} V_k^{i} + \sup_{y \in \fD} \abs*{\bar{w} \pr{y} - \E \br{\bar{w} \pr{y}}}.
    \end{equation}
    Next, observe that the union bound inequality,
    \cref{lem:hoeffding_MC},
    and \eqref{eq:exp_bound_grid_def} assure for all $t \in [0, \infty)$ that
    \begin{equation}
        \begin{split} \label{eq:exp_bound_sup_tail}
            &\mathbb{P} \br*{ \sup_{y \in \fD} \abs*{ \bar{w}\pr{y} - \E \br*{\bar{w}\pr{y}} } \ge \frac{t}{3} } \le \sum_{y \in \fD} 2 \exp\pr*{- \tfrac{n t^2}{ 18 \pr[\big]{\norm{h_0}_{\infty} + \tfrac{M \diam \pr{D}^2 \norm{h_1}_{\infty}  }{d}}^2 }} \\
            &\le 2 \exp\pr*{ d\pr*{\log\pr{K} + 2 \log\pr{M} + M \log\pr{1 + L}} - \tfrac{nt^2}{ 18 \pr[\big]{\norm{h_0}_{\infty} + \tfrac{M \diam \pr{D}^2 \norm{h_1}_{\infty}  }{d}}^2 }}.
        \end{split}
    \end{equation}
    Moreover, note that \cref{lem:tail_probability_times} and the assumption that $V_k^{i}$, $i,k \in \N$, are identically distributed 
    demonstrate for all $t \in [0, \infty)$ that
    \begin{equation}
        \begin{split} \label{eq:exp_bound_V_tail}
            \mathbb{P} \br*{ \tfrac{2L^1_1 \diam \pr{D}^4}{d^{\nicefrac{1}{2}} n K M^2} \SmallSum{i = 1}{n} \SmallSum{k = 1}{M} V_k^{i} \ge \frac{t}{3} } &\le n M \mathbb{P}\br*{V_1^1 \ge \tfrac{t d^{\nicefrac{1}{2}} KM}{6 L_1^1 \diam \pr{D}^4}} \le  4nM \exp\pr*{ - \tfrac{t d^{\nicefrac{3}{2}} K M }{12 L_1^1 \diam \pr{D}^4} }.
        \end{split}
    \end{equation}
    Furthermore, observe that \cref{lem:tail_probability_times} and the assumption that $\rho_k^{i}$, $i,k\in \N$, are identically distributed 
    imply for all $t \in [0, \infty)$ that 
    \begin{equation}
        \begin{split} \label{eq:exp_bound_rho_tail}
            \mathbb{P} \br*{ \tfrac{2\diam \pr{D}^2 \pr{L_1^0 d + L_1^1 \diam \pr{D}^2 M} }{d^{\nicefrac{1}{2}} n K M^2 } \SmallSum{i = 1}{n} \SmallSum{k = 1}{M} \rho_k^{i} \ge \frac{t}{3} } &\le nM \mathbb{P} \br*{ \rho_1^1 \ge \tfrac{t d^{\nicefrac{1}{2}} K M  }{6 \diam \pr{D}^2 \pr{ L_1^0 d + L_1^1 \diam \pr{D}^2 M }} }\\
            &\le 2nM \exp\pr*{ -\tfrac{ t d^{\nicefrac{3}{2}} K M }{ 12 \diam \pr{D}^2 \pr{ L_1^0 d + L_1^1 \diam \pr{D}^2 M } } }.
        \end{split}
    \end{equation}
    Combining \eqref{eq:exp_bound_Z_bound},
    \eqref{eq:exp_bound_sup_tail}, \eqref{eq:exp_bound_V_tail}, and \eqref{eq:exp_bound_rho_tail}
    with the union bound inequality
    establishes for all $t \in [0, \infty)$ that
    \begin{equation}
        \begin{split}
            &\mathbb{P} \pr*{ Z \ge t }\\
            &\le 2 \exp\pr[\bigg]{ d\pr*{\log\pr{K} + 2 \log\pr{M} + M \log\pr{1 + L}} - \tfrac{n{t}^2}{ 18 \pr[\big]{\norm{h_0}_{\infty} + \tfrac{M \diam \pr{D}^2 \norm{h_1}_{\infty}  }{d}}^2 }} \\
            &\quad+ 4nM \exp\pr*{ - \tfrac{t d^{\nicefrac{3}{2}} K M }{12 L_1^1 \diam \pr{D}^4} } + 2nM \exp\pr*{ -\tfrac{ t d^{\nicefrac{3}{2}} K M }{ 12 \diam \pr{D}^2 \pr{ L_1^0 d + L_1^1 \diam \pr{D}^2 M } } } \\
            &\le 2 \bigg( \!\! \exp\pr[\bigg]{ d\pr*{\log\pr{K} + 2 \log\pr{M} + M \log\pr{1 + L}} - \tfrac{n{t}^2}{ 18 \pr[\big]{\norm{h_0}_{\infty} + \tfrac{M \diam \pr{D}^2 \norm{h_1}_{\infty}  }{d}}^2 }} \\
            &\quad+ \exp \pr[\bigg]{\log\pr{3} + \log\pr{n} + \log\pr{M} - \tfrac{ t d^{\nicefrac{3}{2}} K M }{ 12 \diam \pr{D}^2 \pr{ L_1^0 d + L_1^1 \diam \pr{D}^2 M } } } \!\!\bigg).
        \end{split}
    \end{equation}
    This,
    \eqref{eq:exp_bound_sup_error},
    and \cref{lem:expectation_bound_by_tail_bound} (applied with $X \with Z$, $A \with 0$,
    \begin{equation}
        \begin{split}
            &c_1 \with d\pr*{\log\pr{K} + 2 \log\pr{M} + M \log\pr{1 + L}}, \quad c_2 \with \tfrac{n}{18 \pr[\big]{\norm{h_0}_{\infty} + \tfrac{M \diam \pr{D}^2 \norm{h_1}_{\infty}  }{d}}^2}, \\
            &c_3 \with \log\pr{3} + \log\pr{n} + \log\pr{M}, \quad c_4 \with \tfrac{d^{\nicefrac{3}{2}} KM}{12 \diam \pr{D}^2 \pr{ L_1^0 d + L_1^1 \diam \pr{D}^2 M }}
        \end{split}
    \end{equation}
    in the notation of \cref{lem:expectation_bound_by_tail_bound})
    prove \eqref{eq:exp_bound}.
    The proof of \cref{thm:exp_bound} is thus complete.
\end{proof}

\begin{corollary} \label{cor:l_infty}
    Assume the situation of \cref{thm:exp_bound},
    let $\eps \in \pr{0,1}$,
    and assume
    \begin{equation}
        \begin{split}
            \delta &= \min \cu*{1, \tfrac{\eps^2}{9 \pr{\pr{L_1^0 + \norm{h_1}_{\infty}} \diam \pr{D} + L_2^0 d^{\nicefrac{1}{2}} \diam\pr{D}^{\nicefrac{1}{2}}}^2}}, \\
            M &= \max\cu*{ 1, \ceil*{ \tfrac{ \log\pr{6 \pr{ 2 \norm{h_0}_{\infty} d + \diam \pr{D} \norm{h_1}_{\infty} }} - \log\pr{d} + \log\pr{\eps^{-1}} }{\log \pr*{1 + \tfrac{\beta^2 \delta^2}{2 \diam \pr{D}^2}}} } }, \qquad \text{and} \\
            n = K^2 &= \max\cu*{ 1, \ceil{ 9 C(h_0, h_1, d, D, M)^2 \eps^{-2}}^2 },
        \end{split}
    \end{equation}
    where $C(h_0, h_1, d, D, M) \in [0, \infty)$ satisfies
    \begin{equation}
        \begin{split}
            &C(h_0, h_1, d, D, M) \\
            &= \pr*{L_2^0 + \tfrac{2L_1^0 \diam \pr{D}}{d} + \tfrac{1}{d} \pr*{ 2 \diam \pr{D} \norm{h_1}_{\infty} + \diam \pr{D}^2 \pr*{L_2^1 + \tfrac{2L_1^1 \diam \pr{D}}{d}} }} d^{\nicefrac{1}{2}} \diam \pr{D} \\
            &\quad+ \pr*{L_2^0 + \tfrac{2M \diam \pr{D} \norm{h_1}_{\infty} + L_2^1 M \diam \pr{D}^2}{d}} \tfrac{d^{\nicefrac{1}{2}} \diam \pr{D} }{M^2} \\
            &\quad+ 3 \sqrt{2} \pr*{1 + \sqrt{d} + \sqrt{d \pr{ 2 \log \pr{M} + M \log \pr{1 + L} } + \log\pr{2}}} \pr*{ \norm{h_0}_{\infty} + \tfrac{M \diam \pr{D}^2 \norm{h_1}_{\infty}}{d}} \\
            &\quad+ \tfrac{12 \diam \pr{D}^2 \pr{L_1^0 d + L_1^1 \diam \pr{D}^2 M} \pr{\log(6) + 3 + \log \pr{M}}}{d^{\nicefrac{3}{2}} M}.
        \end{split}
    \end{equation}
    Then it holds that $\E \br*{\sup_{x \in D} \abs{w \pr{x} - \bar{w}\pr{x}}} \le \eps$ and there exist $c,q \in \pr{0,\infty}$ such that $\max \cu{n, M} \le c d^{q}\eps^{-q}$, where $q$ is an absolute constant and $c$ depends only on $L$, $L_1^0$, $L_2^0$, $L_1^1$, $L_2^1$, $\norm{h_0}_{\infty}$, $\norm{h_1}_{\infty}$, $\beta^{-1}$, $\diam \pr{D}$, and $d$ and this dependency is at most polynomial.
\end{corollary}


\section{ANN approximations and complexity analysis} \label{sec:ANN}

In this section, we establish ANN approximation results for the approximation of functions of the form \eqref{eq:expectation}.
To this end, we consider fully-connected feed forward ANNs with ReLU activation function as approximating ANNs.
We adopt the notions of e.g., Jentzen et al. \cite{jentzen2023mathematicalintroductiondeeplearning} (see also, e.g., Grohs et al. \cite[Section~2]{GrohsHornungJentzen2019published}, Grohs et al. \cite[Section~3]{grohs2019deep_published}, Ackermann et al. \cite[Section~2]{ackermann2023deep}, and the references therein). In particular, these references establish an ANN calculus on the set of the considered ANNs encompassing operations such as composition, parallelization, and linear combination of (compatible) ANNs, as well es how the size of the ANNs behave under these operations.
For simplicity, we consider in this section ANN approximations of functions $w \colon D \to \R$,
\begin{equation} \label{eq:expectation_boundary}
    w \pr{x} = \E \br{ h\pr{\tau^x, x + B_{\tau^x}} }, \qquad x \in D.
\end{equation}
This corresponds to the special case of \eqref{eq:expectation} with $h_0 = h$ and $h_1 \equiv 0$. We refer to \cref{rem:ANN_discussion} for a discussion on how to include non-zero $h_1$ as well.

In \cref{subsec:ANN_desc}, we introduce the notion of ANNs that is employed in the remainder of this section. 
In \cref{subsec:ANN_wos_mc}, we provide ANN representation and approximation results for the Walk-on-Spheres process introduced in \cref{sec:modified_wos} and for the Monte Carlo estimator studied in \cref{sec:error_mc}.
In \cref{subsec:ANN_approx}, we combine the ANN representations and approximations from \cref{subsec:ANN_wos_mc} with the error analysis established in \cref{sec:error_mc} to obtain an ANN approximation of functions that satisfy \eqref{eq:expectation_boundary} in the $L^\infty$-sense.

\subsection{ANN descriptions} \label{subsec:ANN_desc}
\begin{defi}[ANNs] \label{def:ANN}
    We denote by $\bN$ the set (of all ANNs) given by
    \begin{equation} \label{eq:ANN_set}
        \textstyle
        \bN =
        \textstyle
        \bigcup_{L \in \N} \bigcup_{l_0, l_1, \dots, l_L \in \N} \pr[\big]{\bigtimes_{k = 1}^{L}   \pr{\R^{l_k \times l_{k - 1}} \times \R^{l_k}}}.
    \end{equation}
    Furthermore, for all $L \in \N$, $l_0, \dots, l_L \in \N$,
    \begin{equation}
        \Phi = \pr{ \pr{W_1, B_1}, \dots, \pr{W_L, B_L}} \in \textstyle \pr[\big]{\bigtimes_{k = 1}^L \pr{\R^{l_k \times l_{k - 1}} \times \R^{l_k}}},
    \end{equation}
    and $x_0 \in \R^{l_0}$, \dots, $x_{L - 1} \in \R^{l_{L - 1}}$ with\footnote{For all $d \in \N$, $x = \pr{x_1, \dots, x_d} \in \R^d$ let $\fM \pr{x} \in \R^d$ satisfy $\fM \pr{x} = \pr{ \max\cu{x_1, 0}, \dots, \max\cu{x_d, 0} }$ (multidimensional version of ReLU activation function).} $x_k = \fM \pr{W_k x_{k - 1} + B_k}$, $k \in \N_0 \cap \br{1, L - 1}$, 
    we denote by
    \begin{equation}
        \begin{split}
            \cL \pr{\Phi} &= L \hspace{3.9cm} \text{the length/depth of $\Phi$}, \\
            \cW \pr{\Phi} &= \max_{k \in \N_0 \cap \br{0,L}} l_k \hspace{2.3cm} \text{the maximum width of $\Phi$}, \\
            \paramANN \pr{\Phi} &= \SmallSum{k = 1}{l} l_k \pr{l_{k - 1} + 1} \hspace{1.6cm} \text{the number of parameters of $\Phi$},
        \end{split}
    \end{equation}
    and by $\cR \pr{\Phi} \in C\pr{\R^{l_0}, \R^{l_L}}$ the function which satisfies $\pr{\cR \pr{\Phi}}\pr{x_0} = W_L x_{L - 1} + B_L$.
\end{defi}

\subsection{ANN approximations of modified Walk-on-Spheres algorithms} \label{subsec:ANN_wos_mc}

The following lemma is a version of \cite[Proposition~3]{yarotsky2017error}, providing an ANN approximation of the product function $\R^2 \ni (x, y) \to xy \in \R$ on compact intervals.

\begin{lemma} \label{lem:ANN_product}
    There exist absolute constants $\kappa_1, \kappa_2, \kappa_3 \in [0, \infty)$ such that for all $c \in \pr{0, \infty}$, $\eps \in (0,1]$ there exists $\Pi \in \bN$ such that
    \begin{enumerate}[label=(\roman*)]
        \item it holds that $\cR \pr{\Pi} \in C\pr{\R^2, \R}$,
        \item it holds for all $x,y \in \br{-c, c}$ that $\abs{ \pr{\cR \pr{\Pi}}\pr{x,y} - xy  } \le \eps$, and
        \item it holds that $\max \cu{\cL \pr{\Pi}, \cW \pr{\Pi}} \le \kappa_1 + \kappa_2 \log_2\pr*{\tfrac{1}{\eps}} + \kappa_3 \pr{\log_{2} \pr{c}}^+$.
    \end{enumerate}
\end{lemma}

The next two lemmas establish ANN representations and approximations of all processes appearing in the estimators of \cref{sec:error_mc}. More precisely, \cref{lem:ANN_wos} establishes an ANN architecture that is used to represent the Walk-on-Spheres process and \cref{lem:ANN_wos_time} provides an ANN architecture that is used to approximate the corresponding exit times.
The proofs are consequences of induction as well as composition (see, e.g., \cite[Proposition~2.1.2]{jentzen2023mathematicalintroductiondeeplearning} \cite[Proposition~2.6]{GrohsHornungJentzen2019published}), parallelization (see, e.g., \cite[Lemma~2.2.14]{jentzen2023mathematicalintroductiondeeplearning} \cite[Corollary~2.23]{GrohsHornungJentzen2019published}), and summation (see, e.g., \cite[Proposition~2.26]{GrohsHornungJentzen2019published}) of ANNs. Additionally, in the proof of \cref{lem:ANN_wos_time} the ANN approximation of the product function \cref{lem:ANN_product} is used.
A similar statement for the walk-on-spheres process itself (i.e., to \cref{lem:ANN_wos}) is formulated and proved in \cite[Lemma~3.6, Corollary~3.7]{beznea2022monte}.

\begin{lemma}\label{lem:ANN_wos}
    Let $d \in \N$, 
    let $\mathfrak{r} \in \bN$ satisfy that $\cR \pr{\mathfrak{r}} \in C \pr{\R^d, \R}$ and $\min \cu{\cL \pr{\fr},\cW \pr{\mathfrak{r}}} \ge 2$,
    let $\pr{p_k}_{k \in \N} \subseteq \R^d$.
    Then there exist $\pr{\bX_k}_{k \in \N_0} \subseteq \bN$ such that
    \begin{enumerate}[label=(\roman*)]
        \item \label{it:ANN_wos_realization} it holds for all $k \in \N$, $x \in D$ that $ \cR \pr{\bX_0}, \cR \pr{\bX_k}\in C \pr{\R^d, \R^d}$, $\pr{\cR \pr{\bX_0}} \pr{x} = x$, and
        \begin{equation}
            \begin{split}
                \pr{\cR \pr{\bX_k}} \pr{x} &= \pr{\cR \pr{\bX_{k - 1}}} \pr{x} + \pr{\cR \pr{\fr}} \pr{ \pr{\cR \pr{\bX_{k - 1}}} \pr{x} } p_k,
            \end{split}
        \end{equation}
        and
        \item \label{it:ANN_wos_size} it holds for all $k \in \N$ that $\cL \pr{\bX_0} = 1$, $\cW \pr{\bX_0} = d$,
        \begin{equation}
            \begin{split} \label{eq:ANN_wos_size}
                &\cL \pr{\bX_k} = k \pr{\cL \pr{\fr} - 1} + 1, \qquad \text{and} \qquad \cW \pr{\bX_k} \le 2d + \cW \pr{\fr}.
            \end{split}
        \end{equation}
    \end{enumerate}
\end{lemma}

\begin{lemma} \label{lem:ANN_wos_time}
    Let $d \in \N$, 
    let $D \subseteq \R^d$ be open, bounded,
    let $\mathfrak{r} \in \bN$ satisfy for all $x \in D$ that $\cR \pr{\mathfrak{r}} \in C \pr{\R^d, \R}$, $\min \cu{\cL \pr{\fr}, \cW \pr{\fr}} \ge 2$ and $0 \le \pr{ \cR\pr{\mathfrak{r}} } \pr{x} \le r\pr{x}$, 
    let $\pr{p_k}_{k \in \N} \subseteq \R^d$ satisfy for all $k \in \N$ that $\norm{p_k} = 1$,
    let $\pr{v_k}_{k \in \N} \subseteq  [0, \infty)$,
    and for every $x \in D$ let $\pr{P^x_k}_{k \in \N_0} \subseteq \R^d$, $\pr{T^x_k}_{k \in \N_0} \subseteq [0, \infty)$ satisfy for every $k \in \N$ that $P_0^x = x$, $P_k^x = P_{k - 1}^x + \pr{\cR\pr{\fr}} \pr{P_{k - 1}^x} p_k$, $T^x_0 = 0$ and $T^x_{k} = T^x_{k - 1} + \pr{\pr{\cR\pr{\fr}} \pr{P_{k - 1}^x}}^2 v_k$.
    Then there exist $\pr{\Theta^\eps_k}_{\pr{k, \eps} \in \N_0 \times (0,1]} \subseteq \bN$
    \begin{enumerate}[label=(\roman*)]
        \item \label{it:ANN_wos_time_realization} it holds for all $k \in \N$, $\eps \in (0,1]$, $x \in D$ that $\cR \pr{\Theta^\eps_0}, \cR \pr{\Theta^\eps_k} \in C\pr{\R^d, \R}$, $\pr{\cR \pr{\Theta^\eps_0}} \pr{x} = 0$,
        \begin{equation}
            \begin{split} \label{eq:ANN_wos_time_realization}
                \abs*{ \pr{\cR \pr{\Theta_k^\eps}} \pr{x} - T_k^x } &\le \eps \SmallSum{j = 1}{k} v_j,
            \end{split}
        \end{equation}
        and
        \item \label{it:ANN_wos_time_size} it holds for all $k \in \N$, $\eps \in (0,1]$ that $\cL \pr{\Theta^\eps_0} = 1$, $ \cW \pr{ \Theta_0^\eps } = d$,
        \begin{align} \label{eq:ANN_wos_time_size}
            &\hspace{-0.5cm}\cL \pr{\Theta_k^\eps} \le \kappa_1 + \kappa_2 \log_2 \pr*{\tfrac{1}{\eps}} + \kappa_3 \log_2\pr{\diam \pr{D}} + k \pr{\cL \pr{\fr} - 1}, \qquad \text{and}\\
            &\hspace{-0.5cm}\cW \pr{ \Theta_k^\eps } \le \max\cu{2, d} + k \max\cu*{ \kappa_1 + \kappa_2 \log_2 \pr*{\tfrac{1}{\eps}} + \kappa_3 \log_2\pr{\diam \pr{D}}, 2d + \cW \pr{\fr} }, \nonumber
        \end{align}
        where $\kappa_1, \kappa_2, \kappa_3 \in [0, \infty)$ are the absolute constants from \cref{lem:ANN_product}.
    \end{enumerate}
\end{lemma}

The following lemma provides an ANN approximation of the Monte Carlo estimator \eqref{eq:estimator_intro} in the case of $h_1 \equiv 0$.
Before stating \cref{lem:ANN_MC_f=0}, we briefly explain some of the mathematical objects appearing in \cref{lem:ANN_MC_f=0}.
The ANN $\fr \in \bN$ in \cref{lem:ANN_MC_f=0} denotes a function which determines the radii of the successive spheres on which the Walk-on-Spheres process takes place.
For $x \in D$ and index $i \in \cu{1,\dots, n}$ the sequence $(P_{k}^{x, i})_{k \in \N_0} \subseteq \R^d$ denotes the (deterministic) Walk-on-Spheres process, with start in $x \in D$, unit directions $\pr{p_{k}^{i}}_{k \in \N} \subseteq \R^d$, and means of determining the radii of the successive spheres through the ANN $\fr$.
The sequence $(T^{x, i}_{k})_{k \in \N_0} \subseteq [0, \infty)$ denotes the corresponding exit times of successive spheres of the associated process $(P^{x, i}_{k})_{k \in \N_0}$, where the real number $v_k^{i} \in [0, \infty)$ takes on the role of the sampled exit time of a Brownian motion from the unit sphere.

\begin{lemma} \label{lem:ANN_MC_f=0}
    Let $d, n, M \in \N$, $\eps \in (0,1]$,
    let $D \subseteq \R^d$ be open, bounded,
    let $L \in [0, \infty)$,
    let $\fh, \fr \in \bN$ satisfy for all $x \in \ol{D}$, $s, t \in \R$ that $\cR\pr{\fh} \in C \pr{\R^{d + 1}, \R}$, $\cR\pr{\fr} \in C\pr{\R^d, \R}$, 
    $\abs{\pr{ \cR \pr{\fh}} \pr{s, x} - \pr{ \cR \pr{\fh}} \pr{t, x}} \le L \abs{s - t}$,
    $\min \cu{ \cL \pr{\fr}, \cW \pr{\fr} } \ge 2$, and $0 \le \pr{\cR \pr{\fr}} \pr{x} \le r\pr{x}$,
    for every $i \in \N \cap \br{1,n}$ let $\pr{p_k^i}_{k \in \N}, \subseteq \R^d$ satisfy for all $k \in \N$ that $\norm{p_k^i} = 1$,
    for every $i \in \N \cap \br{1,n}$ let $\pr{v_k^i}_{k \in \N} \subseteq [0, \infty)$,
    and for every $x \in D$, $i \in \N \cap \br{1,n}$ let $\pr{P^{x,i}_k}_{k \in \N_0} \subseteq \R^d$, $\pr{T^{x,i}_k}_{k \in \N_0} \subseteq [0, \infty)$ satisfy for every $k \in \N$ that $P_0^{x, i} = x$, $P_k^{x, i} = P_{k - 1}^{x, i} + \pr{\cR\pr{\fr}} \pr{P_{k - 1}^{x, i}} p_{k}^{i}$, $T^{x, i}_{0} = 0$ and $T^{x, i}_{k} = T^{x,i}_{k - 1} + \pr{\pr{\cR\pr{\fr}} \pr{P_{k - 1}^{x, i}}}^2 v_k^i$.
    Then there exists $\Phi = \Phi_{n, M}^{\fr, \eps} \in \bN$ such that
    \begin{enumerate}[label=(\roman*)]
        \item \label{it:ANN_MC_f=0_approximation} it holds for all $x \in D$ that
        \begin{equation}
            \abs*{\pr{\cR \pr{\Phi}} \pr{x} - \frac{1}{n} \SmallSum{i = 1}{n} \pr{\cR \pr{\fh}} \pr{T_M^{x, i}, P_M^{x, i}} } \le \frac{\eps L}{n} \SmallSum{i = 1}{n} \SmallSum{k = 1}{M} v_k^i
        \end{equation}
        and
        \item \label{it:ANN_MC_f=0_size} it holds that 
        \begin{align}
            \cL \pr{\Phi} &\le \cL \pr{\fh} + \kappa_1 + \kappa_2 \log_2 \pr*{\tfrac{1}{\eps}} + \kappa_3 \log_2\pr{\diam \pr{D}} + M \pr{\cL \pr{\fr} - 1}  \quad \text{and} \nonumber \\
            \cW \pr{\Phi} &\le n \max \Big\{ \cW \pr{\fh}, \max \cu{2, d} + 2d + \cW \pr{\fr} \\
            &\quad+ M \max\!\cu*{ \kappa_1 + \kappa_2 \log_2 \pr*{\tfrac{1}{\eps}} + \kappa_3 \log_2\pr{\diam \pr{D}}, 2d + \cW \pr{\fr} } \!  \Big\}, \nonumber
        \end{align}
        where $\kappa_1, \kappa_2, \kappa_3 \in [0, \infty)$ are the absolute constants from \cref{lem:ANN_product}.
    \end{enumerate}
\end{lemma}

\begin{proof}
    First, observe that \cref{lem:ANN_wos} (applied for every $i \in \N \cap \br{1,n}$ with $d \with d$, $\fr \with \fr$, $\pr{p_k}_{k \in \N} \with \pr{p_k^i}_{k \in \N}$ in the notation of \cref{lem:ANN_wos})
    and \cref{lem:ANN_wos_time} (applied for every $i \in \N \cap \br{1,n}$ with $d \with d$, $D \with D$, $\fr \with \fr$, $\pr{p_k}_{k \in \N} \with \pr{p_k^i}_{ k \in \N }$, $\pr{v_k}_{k \in \N} \with \pr{ v_k^{i} }_{k \in \N}$, $\pr{P^x_k}_{\pr{x, k} \in D \times \N_0} \with \pr{P^{x, i}_k}_{\pr{x, k} \in D \times \N_0} $, $\pr{T^x_k}_{\pr{x, k} \in D \times \N_0} \with \pr{T^{x, i}_k}_{\pr{x, k} \in D \times \N_0} $ in the notation of \cref{lem:ANN_wos_time})
    demonstrate that there exist $\pr{\bX^{i}_k}_{\pr{i, k} \in \pr{\N \cap \br{1, n} \times \N_0}}, \pr{ \Theta^{\eps, i}_k }_{\pr{i, k} \in \pr{\N \cap \br{1,n}} \times \N_0 } \subseteq \bN$ such that
    \begin{enumerate}[label=(\Roman*)]
        \item \label{it:ANN_MC_f=0_wos_space_time_realization} it holds for all $i \in \N \cap \br{1,n}$, $k \in \N_0$, $x \in D$ that $\cR \pr{ \bX_k^{i}} \in C \pr{\R^d, \R^d}$,  $\cR \pr{\Theta_k^{\eps, i}} \in C \pr{\R^d, \R}$,
        \begin{equation}
            \begin{split}
                \pr{\cR \pr{\bX_k^{i}}} \pr{x} = P_k^{x, i} 
                \qquad \text{and} \qquad \abs*{ \pr{\cR \pr{\Theta_k^{\eps, i}}} \pr{x} - T_k^x } \le \eps \SmallSum{j = 1}{k} v_j,
            \end{split}
        \end{equation}
        and
        \item it holds for all $i \in \N \cap \br{1,n}$, $k \in \N$ that $ \cD \pr{\bX_k^i} = \cD \pr{\bX_k^{1}} $, $\cD \pr{\Theta_k^{\eps, i}} = \cD \pr{\Theta_k^{\eps, 1}}$,
        $\cL \pr{\bX_0^i} = \cL \pr{\Theta_0^{\eps, i}} = 1$, $\cW \pr{\bX_0^i} = \cW \pr{\Theta_0^{\eps, i}} = d$, 
        \begin{align}
            \cL \pr{\bX_k^i} &= k \pr{\cL \pr{\fr} - 1} + 1, \qquad \cW \pr{\bX_k^i} \le 2d + \cW \pr{\fr} \\
            \cL \pr{\Theta_k^{\eps, i}} &\le \kappa_1 + \kappa_2 \log_2 \pr*{\tfrac{1}{\eps}} + \kappa_3 \log_2\pr{\diam \pr{D}} + k \pr{\cL \pr{\fr} - 1}, \qquad \text{and}  \nonumber\\
            \cW \pr{\Theta_k^{\eps, i}} &\le \max\cu{2, d} + k \max\cu*{ \kappa_1 + \kappa_2 \log_2 \pr*{\tfrac{1}{\eps}} + \kappa_3 \log_2\pr{\diam \pr{D}}, 2d + \cW \pr{\fr} }, \nonumber
        \end{align}
        where $\kappa_1, \kappa_2, \kappa_3 \in [0, \infty)$ are the absolute constants from \cref{lem:ANN_product}.
    \end{enumerate}
    This
    \cite[Proposition~2.6]{GrohsHornungJentzen2019published},
    \cite[Corollary~2.23]{GrohsHornungJentzen2019published},
    and \cite[Lemma~2.19]{ackermann2023deep}
    establish that there exists $\Phi \in \bN$ such that
    \begin{enumerate}[label=(\Alph*)]
        \item \label{it:ANN_MC_f=0_Phi_realization} it holds for all $x \in D$ that $\cR \pr{\Phi} \in C \pr{\R^d, \R}$ and 
        \begin{equation}
            \pr{ \cR \pr{\Phi}} \pr{x} = \frac{1}{n} \SmallSum{i = 1}{n} \pr{\cR \pr{\fh}} \pr*{ \pr{\cR \pr{\Theta_M^{\eps, i}}} \pr{x} , \pr{\cR \pr{\bX_M^{i}}}\pr{x} }  
        \end{equation}
        and
        \item \label{it:ANN_MC_f=0_Phi_size} it holds that
        \begin{align}
            \cL \pr{\Phi} &\le \cL \pr{\fh} + \kappa_1 + \kappa_2 \log_2 \pr*{\tfrac{1}{\eps}} + \kappa_3 \log_2\pr{\diam \pr{D}}+ M \pr{\cL \pr{\fr} - 1} \quad \text{and} \nonumber\\
            \cW \pr{\Phi} &\le n \max \Big\{ \cW \pr{\fh}, \max \cu{2, d} + 2d + \cW \pr{\fr} \\
            &\quad+ M \max\!\cu*{ \kappa_1 + \kappa_2 \log_2 \pr*{\tfrac{1}{\eps}} + \kappa_3 \log_2\pr{\diam \pr{D}}, 2d + \cW \pr{\fr} } \!  \Big\}. \nonumber
        \end{align}
    \end{enumerate}
    Observe that \cref{it:ANN_MC_f=0_Phi_realization},
    the assumption that for all $x \in \ol{D}$, $s, t \in \R$ it holds that $\abs*{ \pr{ \cR \pr{\fh}} \pr{s, x} - \pr{ \cR \pr{\fh}} \pr{t, x} } \le L \abs{s - t}$,
    \cref{it:ANN_MC_f=0_wos_space_time_realization},
    and the triangle inequality
    show for all $x \in D$ that 
    \begin{equation}
        \begin{split}
            &\abs*{ \pr{ \cR \pr{\Phi}} \pr{x} - \frac{1}{n} \SmallSum{i = 1}{n} \pr{\cR \pr{\fh}} \pr{T_M^{x, i}, P_M^{x, i}}} \le \frac{L}{n} \SmallSum{i = 1}{n} \abs*{ \pr{\cR \pr{\Theta_M^{\eps, i}}} - T^{x,i}_M } \le \frac{\eps L}{n} \SmallSum{i = 1}{n} \SmallSum{k = 1}{M} v_k^{i}.
        \end{split}
    \end{equation}
    Combining this with \cref{it:ANN_MC_f=0_Phi_size} proves \cref{it:ANN_MC_f=0_approximation,it:ANN_MC_f=0_size}.
    The proof of \cref{lem:ANN_MC_f=0} is thus complete.
\end{proof}

\subsection{ANN approximations for solutions of elliptic PDEs} \label{subsec:ANN_approx}

The next theorem is the main result of this section. We establish that under the assumptions stated below expectations of the form \eqref{eq:expectation_boundary} can be approximated by ANNs in the $L^\infty$-sense, where the size of the approximating ANNs is bounded by a polynomial expression depending on the problem inputs.
In \cref{rem:ANN_discussion}, we discuss the implications of \cref{thm:ANN} for overcoming the curse of dimensionality in the approximation of expectations of the form \eqref{eq:expectation_boundary} by ANNs. Moreover, we comment on further generalizations of \cref{thm:ANN}.

\begin{theo}\label{thm:ANN}
    Let $d \in \N$,
    let $D \subseteq \R^d$ be open, bounded, and convex,
    let $\pr{\Omega, \mathcal{F}, \mathbb{P}}$ be a probability space,
    let $B \colon [0, \infty) \times \Omega \to \R^d$ be a standard Brownian motion,
    for every $x \in \overline{D}$ let $\tau^x = \inf \cu{t \in [0, \infty) \colon x + B_t \in \partial D}$,
    let $L_1, L_2 \in \pr{0, \infty}$,
    let $\fh \in \bN$ satisfy for all $s,t \in [0,\infty)$, $x, y \in \ol{D}$ that
    $\cR \pr{\fh} \in C \pr{\R^{d + 1}, \R}$,
    $\sup_{\pr{u, z} \in [0, \infty) \times \ol{D}} \abs{ \pr{\cR \pr{\fh}} \pr{u, z} } < \infty $, and
    \begin{equation}
        \abs*{ \pr{\cR \pr{\fh}} \pr{s, x} - \pr{\cR \pr{\fh}} \pr{t, y}} \le L_1 \abs{s - t} + L_2 \norm{x - y},
    \end{equation}
    let $w \colon D \to \R$ be continuous and satisfy for all $x \in D$ that $w \pr{x} = \E \br{ \pr{\cR\pr{\fh}} \pr{ \tau^x, x + B_{\tau^x} } }$,
    let $\eta \in [0, \infty)$ be an absolute constant,
    let $\bC \in [0, \infty)$,
    and let $\pr{\fr_\eps}_{\eps \in \pr{0,1}} \subseteq \bN$ satisfy for all $x,y \in D$, $\eps \in \pr{0,1}$ that $\cR \pr{\fr_\eps} \in C \pr{\R^d, \R}$, $\paramANN\pr{\fr_\eps} \le \bC \eps^{-\eta}$,
    \begin{equation} \label{eq:ANN_approx_distance}
        \abs{\pr{\cR \pr{\fr_\eps}} \pr{x} -  r\pr{x}} \le \eps, \qquad \text{and} \qquad \abs{\pr{\cR \pr{\fr_\eps}} \pr{x} - \pr{\cR \pr{\fr_\eps}} \pr{y}} \le \bC \norm{x - y}.
    \end{equation}
    Then there exist an absolute constant $q \in [0, \infty)$ and a constant $c \in [0,\infty)$ that depends only on $\diam \pr{D}$, $\sup_{\pr{u, z} \in [0, \infty) \times \ol{D}} \abs{ \pr{\cR \pr{\fh}} \pr{u, z} }$, $L_1$, $L_2$, $\mathcal{P}\pr{\fh}$, $\bC$, and $d$ and this dependency is at most polynomial such that for every $\eps \in \pr{0,1}$ there exists $\Psi_\eps \in \bN$ with
    \begin{equation} \label{eq:ANN}
        \sup_{x \in D} \abs{ w\pr{x} - \pr{ \cR \pr{\Psi_\eps}} \pr{x} } \le \eps \qquad \text{and} \qquad \paramANN \pr{\Psi_\eps} \le c  \eps^{-q}.
    \end{equation}
\end{theo}

\begin{proof}
    Throughout this proof
    let $\bB = \sup_{\pr{u, z} \in [0, \infty) \times \ol{D}} \abs{ \pr{\cR \pr{\fh}} \pr{u, z} }$,
    let $Z^j_k\colon \Omega \to \R^d$, $j,k\in \N$, be standard normally distributed, let $\xi^j_k\colon \Omega \to \R^d$, $j,k\in \N$, satisfy for all $j,k\in \N$ that $\xi^j_k=\|Z^j_k\|^{-1}Z^j_k$,
    let $\zeta\colon \Omega \to [0,\infty)$ satisfy $\zeta=\inf\{t\in [0,\infty)| \|B_t\|\ge 1\}$, let $\rho_k^j$, $k,j\in \N$, be i.i.d.\ copies of $\zeta$, assume that $\rho_k^j$, $Z^j_k$, $j,k\in \N$ are independent,
    let $\pr{\delta_\eps}_{\eps \in \pr{0,1}} \subseteq (0, 1]$ satisfy for all $\eps \in \pr{0,1}$ that 
    \begin{equation} \label{eq:ANN_delta}
        \delta_\eps = \min \cu*{1, \frac{\eps^2}{36 \pr{L_1 \diam \pr{D} + L_2 \diam \pr{D}^{\nicefrac{1}{2}} d^{\nicefrac{1}{2}}}^2} },
    \end{equation}
    let $\pr{M_\eps}_{\eps \in \pr{0,1}} \subseteq \N$ satisfy for all $\eps \in \pr{0,1}$ that
    \begin{equation} \label{eq:ANN_M}
        M_\eps = \max\cu*{1, \ceil*{ \frac{\log \pr{24 \bB} + \log\pr{\eps^{-1}} }{ \log \pr*{ 1 + \frac{\delta_\eps^2}{18 \diam \pr{D}^2  } } } } }, %
    \end{equation}
    let $\pr{\gamma_\eps}_{\eps \in \pr{0,1}} \subseteq (0, 1]$ satisfy for all $\eps \in \pr{0,1}$ that
    \begin{equation} \label{eq:ANN_gamma_product}
        \gamma_\eps = \min \cu*{1, \tfrac{d \eps}{2L_1 M_\eps}},
    \end{equation}
    and let $\pr{n_\eps}_{\eps \in \pr{0,1}} \subseteq \N$ satisfy for all $\eps \in \pr{0,1}$ that
    \begin{equation} \label{eq:ANN_n}
        n_\eps = \max \cu*{1, \ceil*{36 c\pr{d, M_\eps, \fh, D, \bC}^{2} \eps^{-2}}^2}, 
    \end{equation}
    where 
    \begin{equation}
        \begin{split} \label{eq:ANN_n_c}
            \hspace{-0.2cm} c \pr{d, M_{\eps}, \fh, D, \bC} &= \pr*{L_2 + \frac{2 L_1 \diam \pr{D}}{d}} d^{\nicefrac{1}{2}} \diam \pr{D} + L_2 \frac{d^{ \nicefrac{1}{2} } \diam \pr{D}}{M_{\eps}^{2}}\\ 
            &\quad+ 3 \sqrt{2} \bB \pr*{ 1 + d^{\nicefrac{1}{2}} + \sqrt{ d \pr{2 \log \pr{M_{\eps}} + M_{\eps} \log \pr{1 + \bC}} + \log \pr{2} } } \\
            &\quad+ \frac{12 \diam \pr{D}^2 L_1 \pr{\log \pr{6} + 3 + \log \pr{M_{\eps}}}}{d^{\nicefrac{1}{2}} M_{\eps}}.
        \end{split}
    \end{equation}
    Combining \cite[Lemma~3.2]{ackermann2023deep},
    \cite[Proposition~2.6]{GrohsHornungJentzen2019published},
    and \eqref{eq:ANN_approx_distance}
    with \cite[Remark~2.8]{beznea2022monte}
    demonstrates for all $\eps \in \pr{0,1}$ that there exists $\fb_\eps \in \bN$ such that it holds for all $x \in \R^d$ that $\cR \pr{\fb_\eps} \in C\pr{\R^d, \R}$, $\cL \pr{\fb_\eps} = \cL\pr{\fr_{\frac{\delta_{\eps}}{3}}} + 1$, $\cW \pr{\fb_\eps} \le \cW \pr{\fr_\frac{\delta_{\eps}}{3}}$,
    \begin{equation}
        \pr{\cR \pr{\fb_\eps}} \pr{x} = \pr{ \pr{\cR \pr{\fr_{\frac{\delta_{\eps}}{3}}}} \pr{x} - \eps }^+,
    \end{equation}
    and that $\cR \pr{\fb_\eps}$ is a $\pr{\frac{1}{3}, \delta_\eps}$-distance (cf., \cref{def:beta_eps_distance}).
    Subsequently, for $\eps \in \pr{0,1}$ we denote by $(\bar{X}^{x, j, \eps}_k)_{\pr{x, k, j} \in D \times \N_0 \times \N}$ the (family of) modified Walk-on-Spheres process and by $\pr{\bar{\tau}^{x, j, \eps}_k}_{\pr{x, k, j} \in D \times \N_0 \times \N}$ the (family of) associated (physical) times generated by the $\pr{\frac{1}{3}, \delta_\eps}$-distance $\cR \pr{\fb_\eps}$ as well as the random variables $\xi_k^i, \rho_k^i$, $i, k \in \N$ (cf., \cref{sec:modified_wos}).
    Moreover, for $\eps \in \pr{0,1}$ let $\bar{w}_\eps \colon D \times \Omega \to \R$ satisfy for all $x \in D$ that 
    \begin{equation}
        \bar{w}_\eps \pr{x} = \frac{1}{n_\eps} \sum_{i = 1}^{n_\eps} \pr{\cR \pr{\fh}} \pr{ \bar{\tau}^{x, i, \eps}_{M_\eps}, \bar{X}^{x, i, \eps}_{M_\eps} }.
    \end{equation}

    \medskip \noindent
    Next, observe that \cref{lem:ANN_MC_f=0} (applied for every $\omega \in \Omega$, $\eps \in \pr{0,1}$ with
    $d \with d$, $n \with n_\eps$, $M \with M_\eps$, $\eps \with \gamma_\eps$, $D \with D$, $L \with L_1$, $\fh \with \fh$, $\fr \with \fb_\eps$, $\pr{p_k^i}_{\pr{i, k} \in (\N \cap \br{1,n}) \times \N} \with \pr{\xi_k^i \pr{\omega}}_{\pr{i, k} \in (\N \cap \br{1,n}) \times \N}$, $\pr{v_k^i}_{\pr{i, k} \in (\N \cap \br{1,n}) \times \N} \with \pr{\rho_k^i \pr{\omega}}_{\pr{i, k} \in (\N \cap \br{1,n}) \times \N}$ in the notation of \cref{lem:ANN_MC_f=0})
    demonstrates that for every $\eps \in \pr{0, 1}$ there exists $\Phi_\eps \colon \Omega \to \bN$ 
    such that 
    \begin{enumerate}[label=(\Roman*)]
        \item \label{it:ANN_approx} it holds for all $\omega \in \Omega$, $x \in D$ that
        \begin{equation}
            \abs*{ \pr{ \cR \pr{ \Phi_\eps \pr{\omega} } } \pr{x} - \bar{w}_\eps\pr{x, \omega}  } \le \frac{\gamma_\eps L_1}{n_\eps} \SmallSum{i = 1}{n_\eps} \SmallSum{k = 1}{M_\eps} \rho_k^{i} \pr{\omega}
        \end{equation}
        and
        \item \label{it:ANN_size}it holds for all $\omega \in \Omega$ that
        \begin{equation}
            \begin{split}
                \cL \pr{\Phi_\eps \pr{\omega}} &\le \cL \pr{\fh} + \kappa_1 + \kappa_2 \log_2 \pr*{\tfrac{1}{\gamma_{\eps}}} + \kappa_3 \log_2\pr{\diam \pr{D}} + M_\eps \pr{\cL \pr{\fb_\eps} - 1}  \quad \text{and} \\
                \cW \pr{\Phi_\eps \pr{\omega}} &\le n_\eps \max \Big\{ \cW \pr{\fh}, \max \cu{2, d} + 2d + \cW \pr{\fb_\eps} \\
                &\quad+ M_\eps \max\!\cu*{ \kappa_1 + \kappa_2 \log_2 \pr*{\tfrac{1}{\gamma_{\eps}}} + \kappa_3 \log_2\pr{\diam \pr{D}}, 2d + \cW \pr{\fb_\eps} } \!  \Big\},
            \end{split}
        \end{equation}
        where $\kappa_1, \kappa_2, \kappa_3 \in [0, \infty)$ are the absolute constants from \cref{lem:ANN_product}.
    \end{enumerate}
    This,
    \eqref{eq:ANN_gamma_product},
    the assumption that $\pr{\rho_k^i}$, $i,k \in \N$, are identically distributed,
    and the fact that $\E \br{ \rho_1^1} = \frac{1}{d}$ (see \cite[e.g., Proposition~2.2.21 or 3.1.8]{port1978brownian}) 
    imply that for all $\eps \in \pr{0,1}$ it holds that
    \begin{equation} \label{eq:ANN_ANN_approx}
        \E \br*{ \sup_{x \in D} \abs*{ \pr{\cR \pr{\Phi_\eps}} \pr{x} - \bar{w}_\eps \pr{x}  }} \le \frac{\gamma_\eps L_1 M_\eps}{d} = \frac{\eps}{2}.
    \end{equation}
    Next, note that \cref{cor:l_infty} (applied for every $\eps \in \pr{0, 1}$ with
    $\eps \with \frac{\eps}{2}$, $h_0 \with \cR \pr{\fh}$, $L_1^0 \with L_1$, $L_2^0 \with L_2$, $h_1 \with 0$, $L_1^1 \with 0$, $L_2^1 \with 0$, $L \with L$, $\mdist \with \cR\pr{\fb_\eps}$ in the notation of \cref{cor:l_infty})
    yields for all $\eps \in \pr{0, 1}$ that
    \begin{equation}
        \E \br*{ \sup_{x \in D} \abs*{ w\pr{x} - \bar{w}_\eps\pr{x} } } \le \frac{\eps}{2}.
    \end{equation}
    This,
    \eqref{eq:ANN_ANN_approx},
    and the triangle inequality yield for all $\eps \in \pr{0,1}$ that
    \begin{equation}
        \begin{split}
            &\E \br*{ \sup_{x \in D} \abs{ w\pr{x} - \pr{ \cR \pr{\Phi_\eps}} \pr{x} } } \le \eps.           
        \end{split}
    \end{equation}
    This implies that there exists $\omega_\eps \in \Omega$, $\eps \in \pr{0,1}$, such that for all $\eps \in \pr{0,1}$ it holds that
    \begin{equation} \label{eq:ANN_deterministic_error}
        \sup_{x \in D} \abs*{ w\pr{x} - \pr{ \cR \pr{ \Phi_\eps \pr{\omega_\eps} } } \pr{x} } \le \eps.
    \end{equation}
    Next, observe that \eqref{eq:ANN_delta}
    and Jensen's inequality
    ensure for all $\eps \in \pr{0,1}$ that
    \begin{equation}
        \begin{split} \label{eq:ANN_delta_upper_bound}
            \delta_\eps^{-1} &= \max\cu*{1, 36 \pr{L_1 \diam \pr{D} + L_2 \diam\pr{D}^{\nicefrac{1}{2}} d^{\nicefrac{1}{2}} }^2 \eps^{-2}} \\ 
            &\le \pr{1 + 72 \diam \pr{D} \pr{L_1^2 \diam \pr{D} + L_2^2}} d \eps^{-2}.
        \end{split}
    \end{equation}
    This,
    \eqref{eq:ANN_M},
    the fact that for all $x \in \pr{0,\infty}$ it holds that $\log\pr{x} \le x$, 
    and the fact that for all $x\in \pr{-1, \infty}$ it holds that $\log \pr{1 + x} \ge \frac{x}{1 + x}$
    demonstrate for all $\eps \in \pr{0,1}$ that
    \begin{equation}
        \begin{split} \label{eq:ANN_M_upper_bound}
            M_\eps &= \max\cu*{1, \ceil*{ \frac{\log \pr{24 \bB} + \log\pr{\eps^{-1}} }{ \log \pr*{ 1 + \frac{\delta_\eps^2}{18 \diam \pr{D}^2  } } } } }, \\
            &\le 2 + 18 \pr{\log \pr{24 \bB} + \log\pr{\eps^{-1}}} \diam \pr{D}^2 \delta_\eps^{-2} + \log \pr{24 \bB} + \log\pr{\eps^{-1}} \\
            &\le c_1 \pr{\fh, D} d^2 \eps^{-5},
        \end{split}
    \end{equation}
    where
    \begin{equation}
        \begin{split}
            c_1 \pr{\fh, D} &= 3 + \log \pr{ 24 \bB }\\
            &\quad+ 18 \pr{ \log \pr{ 24 \bB} + 1} \diam \pr{D}^2 \pr{1 + 72 \diam \pr{D}\pr{L_1^2 \diam \pr{D} + L_2^2}}^2
        \end{split}
    \end{equation}
    This, the fact that for all $x \in \pr{0,\infty}$ it holds that $\log_2\pr{x} \le x$, and
    \eqref{eq:ANN_gamma_product}
    imply for all $\eps \in \pr{0,1}$ that
    \begin{equation} \label{eq:ANN_gamma_product_upper_bound}
        \log_2 \pr{\gamma_{\eps}^{-1}} \le \gamma_\eps^{-1} = \max\cu*{1, \tfrac{2L_1 M_\eps}{d \eps}} \le \pr{1 + 2L_1 c_1 \pr{\fh, D}} d  \eps^{-6}.
    \end{equation}
    Furthermore, observe that \eqref{eq:ANN_n_c} and the fact that for all $\eps \in \pr{0,1}$ it holds that $d, M_{\eps} \in \N$
    assure for all $\eps \in \pr{0,1}$ that
    \begin{equation}
        c\pr{d, M_\eps, \fh, D} \le c_2 \pr{\fh, D, \bC} M_{\eps}^{\nicefrac{1}{2}} d^{\nicefrac{1}{2}},
    \end{equation}
    where 
    \begin{equation}
        \begin{split}
            c_2 \pr{\fh, D, \bC} &= 2 \diam \pr{D} \pr[\big]{L_2 + 6 \diam \pr{D} L_1 \pr{1 + \log \pr{6}} } \\
            &\quad+ 3 \sqrt{2} \bB \pr[\big]{ 2 + \sqrt{2 + \log\pr{2} + \log \pr{1 + \bC}} }.
        \end{split}
    \end{equation}
    This,
    \eqref{eq:ANN_n},
    and Jensen's inequality establish for all $\eps \in \pr{0,1}$ that
    \begin{equation}
        \begin{split}
            n_{\eps} &\le 1 + \pr{1 + 36 c(d, M_{\eps}, \fh, D, \bC)^{2} \eps^{-2} }^2 \le 3 + 2592 c_2 \pr{\fh, D, \bC}^4 M_{\eps}^2 d^{2} \eps^{-4}. 
        \end{split}
    \end{equation}
    This and \eqref{eq:ANN_M_upper_bound} show for all $\eps \in \pr{0,1}$ that
    \begin{equation} \label{eq:ANN_n_upper_bound}
        n_\eps \le c_3 \pr{\fh, D, \bC} d^{6} \eps^{-14},
    \end{equation}
    where $c_3 \pr{\fh, D, \bC} =  3 + 2592 c_{1}\pr{\fh, D}^2 c_{2} \pr{\fh, D, \bC}^4$.
    Combining this, \eqref{eq:ANN_deterministic_error}, 
    \cite[Lemma~3.3]{jentzen2025dnnControl},
    the assumption that for all $\eps \in \pr{0,1}$ it holds that $\paramANN\pr{\fr_\eps} \le \bC \eps^{-\eta}$,
    \eqref{eq:ANN_delta_upper_bound},
    \eqref{eq:ANN_M_upper_bound},
    and \eqref{eq:ANN_gamma_product_upper_bound}
    with \cref{it:ANN_size} proves \eqref{eq:ANN}.
    The proof of \cref{thm:ANN} is thus complete.
\end{proof}

\begin{remark}\label{rem:ANN_discussion}
    We conclude this work with a few remarks concerning the conclusion of \cref{thm:ANN} as well as ways to extend the result.
    \begin{enumerate}[label=(\roman*)]
        \item Observe that \cref{thm:ANN} states an explicit polynomial dependence of the constant $c$ in its conclusion on the problem dimension $d$. Moreover, additional implicit dependencies of $c$ on $d$ may arise through the remaining problem parameters appearing in the theorem such as for example $\diam \pr{D}$.
        \cref{thm:ANN} implies that under the assumptions that these terms grow at most polynomially in the dimension $d \in \N$, i.e., there exists an absolute constant $\bK \in [0, \infty)$ such that for all $T \in \cu{\diam \pr{D}, \sup_{\pr{u, z} \in [0, \infty) \times \ol{D}} \abs{ \pr{\cR \pr{\fh}} \pr{u, z} }, L_1, L_2, \mathcal{P} \pr{\fh}, \bC}$ it holds that $T \le \bK d^{\bK}$, the overall number of parameters of the approximating ANNs grows at most polynomially in the dimension $d$ and the reciprocal $\eps^{-1}$ of the approximation accuracy $\eps$. In this sense, the ANN approximation presented in \cref{thm:ANN} overcomes the curse of dimensionality.
        \item The assumption that $h$ can be represented directly as an ANN can be weakened to the assumption that $h$ can be approximated uniformly on $[0, \infty) \times \ol{D}$ the size of these approximating ANNs appears in the dependency of the constant $c$ in the conclusion of \cref{thm:ANN}.
        \item Note that \cref{lem:ANN_wos,lem:ANN_wos_time,lem:ANN_MC_f=0} can be extended to include non-zero $h_1$ under suitable assumptions as imposed on $h_0$ (i.e., $\fh$) in the approximation of the Monte Carlo scheme \eqref{eq:estimator_intro}. The statement and proof of \cref{thm:ANN} can be extended to this scenario as well.
    \end{enumerate}
\end{remark}


\bibliographystyle{plain}
\bibliography{literature}


\end{document}